\documentclass[12pt]{amsart}
\usepackage{latexsym,amsmath,amssymb,epsfig,amsthm}
\usepackage{graphicx}
\usepackage{tikz}
\usepackage[enableskew,vcentermath]{youngtab}
\usepackage{hyperref}
\hypersetup{colorlinks=false}
\input epsf
\newtheorem{theorem}{Theorem}[section]
\newtheorem{proposition}[theorem]{Proposition}
\newtheorem{corollary}[theorem]{Corollary}

\newtheorem{example}[theorem]{Example}
\newtheorem{lemma}[theorem]{Lemma}
\newtheorem{remark}[theorem]{Remark}
\newtheorem{problem}[theorem]{Problem}
\newtheorem{question}[theorem]{Question}
\newcommand{\asc}{{\rm asc}}
\newcommand{\Asc}{{\rm Asc}}
\newcommand{\ch}{{\rm ch}}
\newcommand{\des}{{\rm des}}
\newcommand{\Des}{{\rm Des}}
\newcommand{\esd}{{\rm esd}}

\newcommand{\exc}{{\rm exc}}
\newcommand{\fexc}{{\rm fexc}}

\newcommand{\ree}{{\rm e}}
\newcommand{\sDes}{{\rm sDes}}

\newcommand{\supp}{{\rm supp}}

\newcommand{\bB}{{\mathcal B}}
\newcommand{\dD}{{\mathcal D}}

\newcommand{\fF}{{\mathcal F}}
\newcommand{\lL}{{\mathcal L}}
\newcommand{\pP}{{\mathcal P}}
\newcommand{\qQ}{{\mathcal Q}}

\newcommand{\bx}{{\mathbf x}}
\newcommand{\by}{{\mathbf y}}
\newcommand{\RR}{{\mathbb R}}
\newcommand{\fS}{{\mathfrak S}}
\newcommand{\NN}{{\mathbb N}}
\newcommand{\ZZ}{{\mathbb Z}}
\newcommand{\CC}{{\mathbb C}}
\renewcommand{\to}{\rightarrow}

\newcommand{\sm}{{\smallsetminus}}
\begin{document}
\title[Binomial Eulerian polynomials for colored 
permutations]
{Binomial Eulerian polynomials for colored 
permutations}

\author{Christos~A.~Athanasiadis}

\address{Department of Mathematics\\
National and Kapodistrian University of Athens\\
Panepistimioupolis\\
15784 Athens, Greece}
\email{caath@math.uoa.gr}

\date{December 1, 2018; revised on January 20, 2020}
\thanks{ \textit{Key words and phrases}. 
Eulerian polynomial, colored permutation, triangulation,
$h$-polynomial, face ring, local face module, equivariant 
$\gamma$-positivity.}

\begin{abstract}
Binomial Eulerian polynomials first appeared in
work of Postnikov, Reiner and Williams on the face 
enumeration of generalized permutohedra. They are 
$\gamma$-positive (in particular, palindromic and 
unimodal) polynomials which can be interpreted as 
$h$-polynomials of certain flag simplicial polytopes  
and which admit interesting Schur $\gamma$-positive 
symmetric function generalizations. This paper 
introduces analogues of these polynomials for 
$r$-colored permutations with similar properties 
and uncovers some new instances of equivariant 
$\gamma$-positivity in geometric combinatorics.
\end{abstract}

\maketitle

\section{Introduction}
\label{sec:intro}

The Eulerian polynomial
\[ A_n(t) \ := \ \sum_{w \in \fS_n} t^{\des(w)} 
        \ = \ \sum_{w \in \fS_n} t^{\exc(w)}, \]
attached to the symmetric group $\fS_n$ of 
permutations of the set $\{1, 2,\dots,n\}$, 
is a prototypical example of a palindromic and unimodal 
polynomial in combinatorics (for undefined notation 
and terminology, see Sections~\ref{sec:pre}, 
\ref{sec:trian} and~\ref{sec:act}). Its palindromicity 
and unimodality can be demonstrated combinatorially 
and geometrically (as well as by other approaches). 
For instance, these properties follow from an expansion 
of the form 
\[ A_n(t) \ = \sum_{i=0}^{\lfloor (n-1)/2 \rfloor} 
  \gamma_{n,i} \, t^i (1+t)^{n-1-2i} \]
for some nonnegative integers $\gamma_{n,i}$, which
establishes the $\gamma$-positivity of $A_n(t)$, and 
from the interpretation of $A_n(t)$ as the $h$-polynomial 
of the boundary complex of a simplicial polytope. For an 
overview of these and other important properties of 
Eulerian polynomials, see \cite{Pet15} 
\cite[Section~1.4]{StaEC1}.

The $n$th \emph{binomial Eulerian polynomial}, defined 
by the formula
\begin{equation}
\label{eq:defbinomAn}
\widetilde{A}_n(t) \ = \ 1 \, + \, t \, \sum_{m=1}^n 
                          {n \choose m} A_m(t), 
\end{equation}
first appeared in work of Postnikov, Reiner and 
Williams~\cite[Section~10.4]{PRW08} on the face 
enumeration of generalized permutohedra and was further 
studied in~\cite{SW17+} (where the name binomial 
Eulerian polynomial was adopted) and in~\cite{MMY17}.
The following statement, which can be derived from a 
more general result \cite[Theorem~11.6]{PRW08} on the 
$h$-polynomials of chordal nestohedra, shows that 
$\widetilde{A}_n(t)$ shares some of the main properties
of the Eulerian polynomial $A_n(t)$.
\begin{theorem} \label{thm:mainA} 
{\rm (Postnikov--Reiner--Williams~\cite{PRW08})}
The polynomial $\widetilde{A}_n(t)$ is equal to the 
$h$-polynomial of the boundary complex of an 
$n$-dimensional flag simplicial polytope for every 
positive integer $n$. Moreover,
\begin{equation}
\label{eq:binomAngamma}
  \widetilde{A}_n(t) \ = \ 
  \sum_{i=0}^{\lfloor n/2 \rfloor} 
  \widetilde{\gamma}_{n,i} \, t^i (1+t)^{n-2i},
\end{equation} 
where $\widetilde{\gamma}_{n,i}$ is equal to the 
number of permutations $w \in \fS_{n+1}$ which have  
exactly $i$ descents, no two consecutive, and $w(1) 
< w(2) < \cdots < w(m) = n+1$ for some $m$.

In particular, $\widetilde{A}_n(t)$ is 
$\gamma$-positive.
\end{theorem}

This paper introduces a generalization 
$\widetilde{A}_{n,r}(t)$ of $\widetilde{A}_n(t)$ to 
the wreath product group $\ZZ_r \wr \fS_n$ and 
studies further symmetric function generalizations. 
This section gives only a rough outline of the paper; 
for background, precise statements of most of the 
results and technicalities the reader is referred to 
other sections. The polynomial $\widetilde{A}_{n,r}
(t)$ is defined by the formula
\begin{equation}
\label{eq:defbinomAnr}
\widetilde{A}_{n,r}(t) \ = \ \sum_{m=0}^n 
{n \choose m} t^{n-m} A_{m,r}(t), 
\end{equation}
where
\[ A_{n,r}(t) \ := \sum_{w \in \ZZ_r \wr \fS_n} 
                t^{\des(w)} 
           \ = \sum_{w \in \ZZ_r \wr \fS_n} 
                t^{\exc(w)} \]

\medskip
\noindent
is the Eulerian polynomial for $\ZZ_r \wr \fS_n$ 
introduced by Steingr\'imsson \cite{Stei92, Stei94} 
and $A_{0,r}(t) := 1$. Since, by the palindromicity 
of $A_n(t)$ and $\widetilde{A}_n(t)$, Equation 
(\ref{eq:defbinomAn}) may be rewritten as 
\begin{equation}
\label{eq:def2binomAn}
\widetilde{A}_n(t) \ = \ \sum_{m=0}^n {n \choose m} 
                         t^{n-m} A_m(t),
\end{equation}
where $A_0(t) := 1$, the polynomial $\widetilde{A}_{n,r}
(t)$ reduces to $\widetilde{A}_n(t)$ for $r=1$. For 
$r=2$, we prefer to write $\widetilde{B}_n (t)$, instead 
of $\widetilde{A}_{n,2}(t)$, to parallel to standard 
notation $B_n(t)$ for the Eulerian polynomial 
$A_{n,2}(t)$ associated to the hyperoctahedral group 
$\ZZ_2 \wr \fS_n$. For the first few values of $n$, we 
have
\[  \widetilde{B}_n (t) \ = \ \begin{cases}
    1 + 2t, & \text{if $n=1$} \\
    1 + 8t + 4t^2, & \text{if $n=2$} \\
    1 + 26t + 44t^2 + 8t^3, & \text{if $n=3$} \\
    1 + 80t + 328t^2 + 208t^3 + 16t^4, & 
                             \text{if $n=4$} \\
    1 + 242t + 2072t^2 + 3072t^3 + 912t^4 + 32t^5, 
                               & \text{if $n=5$}. 
    \end{cases} \]

The type of palindromic decomposition of 
$\widetilde{A}_{n,r}(t)$ provided by the following 
statement has strong implications. For instance, it 
shows that 
$\widetilde{A}_{n,r}(t)$ is alternatingly increasing, 
in the sense of \cite[Definition~2.9]{SV13}, and 
nonsymmetric $\gamma$-positive, in the sense of 
\cite[Section~5.1]{Ath17}, and hence unimodal, for all 
$n \ge 1$ and $r \ge 2$ and thus it adds another family 
of polynomials in the literature known to have these nice 
properties. Palindromic decompositions of polynomials 
have been considered before in the context of Ehrhart 
theory \cite{Stap09}; see also \cite{BS18} for 
connections to real-rootedness and 
\cite[Sections~10.3-10.4]{BR15}.

\begin{theorem} \label{thm:mainB} 
For all positive integers $n, r$ with $r \ge 2$ 
we have
\begin{equation}
\label{eq:binomAnrsum}
\widetilde{A}_{n,r}(t) \ = \ 
\widetilde{A}^+_{n,r}(t) \, + \, 
\widetilde{A}^-_{n,r}(t),
\end{equation}
where 

\begin{itemize}
\itemsep=0pt

\item
$\widetilde{A}^+_{n,r}(t)$ is a $\gamma$-positive 
polynomial with center $n/2$, and 

\item
$\widetilde{A}^-_{n,r}(t)$ is a $\gamma$-positive 
polynomial with center $(n+1)/2$ and zero constant 
term. 
\end{itemize}

Moreover, $\widetilde{A}^+_{n,r}(t)$ equals the 
$h$-polynomial of the boundary complex of an 
$n$-dimensional flag simplicial polytope.
\end{theorem}

For instance, for $r=2$, writing $\widetilde{B}^+_n 
(t)$ and $\widetilde{B}^-_n (t)$ instead of 
$\widetilde{A}^+_{n,2}(t)$ and $\widetilde{A}^-_{n,2}
(t)$, we have

\[  \widetilde{B}^+_n (t) \ = \ \begin{cases}
    1 + t, & \text{if $n=1$} \\
    1 + 5t + t^2, & \text{if $n=2$} \\
    1 + 19t + 19t^2 + t^3, & \text{if $n=3$} \\
    1 + 65t + 185t^2 + 65t^3 + t^4, & 
                            \text{if $n=4$} \\
    1 + 211t + 1371t^2 + 1371t^3 + 211t^4 + t^5, 
                           & \text{if $n=5$} 
    \end{cases}  \]
and
\[  \widetilde{B}^-_n (t) \ = \ \begin{cases}
    t, & \text{if $n=1$} \\
    3t + 3t^2, & \text{if $n=2$} \\
    7t + 25t^2 + 7t^3, & \text{if $n=3$} \\
    15t + 143t^2 + 143t^3 + 15t^4, & 
                         \text{if $n=4$} \\
    31t + 701t^2 + 1701t^3 + 701t^4 + 31t^5, 
                          & \text{if $n=5$}. 
    \end{cases} \]

\medskip
The Eulerian polynomial $A_n(t)$ affords an 
$\fS_n$-equivariant analogue, meaning a graded 
$\fS_n$-representation $\varphi_n = \oplus_{j=0}^{n-1} 
\, \varphi_{n,j}$ such that $\sum_{j=0}^{n-1} \dim 
(\varphi_{n,j}) t^j = A_n(t)$, which arises very 
often in mathematics: most importantly, in the 
contexts of equivariant cohomology of toric varieties 
\cite{Pro90} \cite[p.~529]{Sta89} \cite{Ste92}, group 
actions on face rings of simplicial complexes 
\cite{Ste92} and the homology of posets \cite{SW09}, 
equivariant Ehrhart theory \cite[Section~9]{Stap11}, 
as well as in purely enumerative contexts; see 
\cite{SW16} for an overview. The representation 
$\varphi_n$ can also be determined by the generating 
function formula 
\begin{equation} \label{eq:genMn}
  1 \, + \, \sum_{n \ge 1} z^n \,
  \sum_{j=0}^{n-1} \ch (\varphi_{n,j}) (\bx) t^j \ = 
  \ \frac{(1-t) H(\bx; z)}{H(\bx; tz) - tH(\bx; z)},
\end{equation}
where $\ch$ stands for Frobenius characteristic and
$H(\bx; z) = \sum_{n \ge 0} h_n(\bx) z^n$ is the 
generating function for the complete homogeneous
symmetric functions in $\bx = (x_1, x_2,\dots)$. An
$\fS_n$-equivariant analogue $\widetilde{\varphi}_n = 
\oplus_{j=0}^n \, \widetilde{\varphi}_{n,j}$ of the 
binomial Eulerian polynomial $\widetilde{A}_n(t)$, 
satisfying 
\begin{equation} \label{eq:genTMn}
  1 \, + \, \sum_{n \ge 1} z^n \, \sum_{j=0}^n 
  \ch (\widetilde{\varphi}_{n,j}) (\bx) t^j \ = \ 
  \frac{(1-t) H(\bx; z) H(\bx; tz)} 
  {H(\bx; tz) - tH(\bx; z)},
\end{equation}
was found by Shareshian and Wachs~\cite{SW17+}, 
who computed the graded $\fS_n$-representation on 
the cohomology of the toric variety associated to 
the polytope mentioned in Theorem~\ref{thm:mainA},
namely the $n$-dimensional simplicial stellohedron. 
As shown in \cite[Section~3]{SW17+}, the coefficients 
of $z^n$ in the right-hand sides of 
Equations~(\ref{eq:genMn}) and~(\ref{eq:genTMn})
are Schur $\gamma$-positive polynomials in $t$, 
meaning that
\begin{eqnarray} \label{eq:phi-n}
\sum_{j=0}^{n-1} \ch (\varphi_{n,j}) (\bx) t^j
& = & \sum_{i=0}^{\lfloor (n-1)/2 \rfloor} 
\gamma_{n,i} (\bx) \, t^i (1+t)^{n-1-2i} \\
& & \nonumber \\ \label{eq:tphi-n}
\sum_{j=0}^n \ch (\widetilde{\varphi}_{n,j}) (\bx) 
t^j & = & \sum_{i=0}^{\lfloor n/2 \rfloor} 
\widetilde{\gamma}_{n,i} (\bx) \, t^i (1+t)^{n-2i} 
\end{eqnarray}
for some Schur-positive symmetric functions 
$\gamma_{n,i} (\bx)$ and $\widetilde{\gamma}_{n,i} 
(\bx)$. These results constitute two of only few 
instances of equivariant $\gamma$-positivity (and more 
specifically, in this case, of the equivariant Gal 
phenomenon introduced in \cite[Section~5]{SW17+}) 
that exist in the literature. This paper contributes 
a few more. For instance, we provide an 
$\fS_n$-equivariant analogue
$\widetilde{\varphi}_{n,r}^+ = \oplus_{j=0}^n \, 
\widetilde{\varphi}_{n,r,j}^+$ of 
$\widetilde{A}^+_{n,r}(t)$, satisfying
\begin{equation} \label{eq:genTMnr}
  1 \, + \, \sum_{n \ge 1} z^n \, \sum_{j=0}^n 
  \ch (\widetilde{\varphi}_{n,r,j}^+) (\bx) t^j \ = 
	\ \frac{H(\bx; z) H(\bx; tz)^r - tH(\bx; z)^r 
	H(\bx; tz)} {H(\bx; tz)^r - tH(\bx; z)^r},
\end{equation}
interpret it as the graded $\fS_n$-representation on 
the cohomology of the toric variety associated to 
the simplicial polytope mentioned in 
Theorem~\ref{thm:mainB} and prove its equivariant 
$\gamma$-positivity for $r = 2$.

The content and other results of this paper may be 
summarized as follows. Section~\ref{sec:pre} includes 
preliminaries on colored permutations, symmetric 
functions and $\gamma$-positivity and fixes 
notation. The first part of Theorem~\ref{thm:mainB} 
is derived in Section~\ref{sec:mainB} from the 
$\gamma$-positivity of the derangement polynomials 
for $r$-colored permutations, established 
in~\cite{Ath14}. The proof yields a combinatorial 
interpretation of the corresponding 
$\gamma$-coefficients (Corollary~\ref{cor:mainB}) 
which generalizes the one provided by 
Theorem~\ref{thm:mainA} in the case $r=1$.

The second statement in Theorem~\ref{thm:mainB} is 
proven in Section~\ref{sec:be}, after the relevant 
background on simplicial complexes and their face 
enumeration is explained in Section~\ref{sec:trian}. 
The proof is based on a natural construction 
\cite{Ath12} of a flag triangulation $\Delta(\Gamma)$ 
of the sphere, which contains as a subcomplex a given 
flag triangulation $\Gamma$ of the simplex of the same 
dimension. This construction is applied in 
Section~\ref{sec:be} to the refinement of the 
barycentric subdivision of the simplex which was 
introduced in \cite{Ath14} to provide a 
partial geometric interpretation for the $r$-colored 
derangement polynomial. The simplicial polytope 
constructed for $r=1$, in which case $\Gamma$ is the 
barycentric subdivision of the $(n-1)$-dimensional 
simplex itself, turns out to be combinatorially 
isomorphic to the $n$-dimensional simplicial 
stellohedron, considered in~\cite{PRW08} (see 
Example~\ref{ex:h}). 

The remainder of this paper focuses on equivariant 
analogues of polynomials discussed so far. 
Section~\ref{sec:act} reviews standard 
ways in which equivariant analogues of polynomials 
of combinatorial significance arise in the 
literature, namely from group 
actions on simplicial complexes, fans, posets and 
lattice points in polyhedra. 
Sections~\ref{sec:localgal} and~\ref{sec:gal} provide 
$\fS_n$-equivariant analogues of $r$-colored 
derangement and binomial Eulerian polynomials, 
partially interpret them in terms 
of $\fS_n$-actions on face rings of triangulations 
of spheres or local face modules of triangulations 
of simplices and partially prove their (nonsymmetric)
equivariant $\gamma$-positivity. Thus, some new 
instances of the equivariant Gal phenomenon 
\cite[Section~5]{SW17+} and its local version 
\cite[Section~5.2]{Ath17} \cite[Section~5]{Ath18} are
uncovered.
The methods employed in these sections extend those 
of \cite[Section~5]{Ath18} and \cite{Ste94}. A 
combinatorial interpretation of the 
$\gamma$-coefficients of $\widetilde{B}^+_n (t)$ 
and $\widetilde{B}^-_n (t)$ which is different 
from the one given in Section~\ref{sec:mainB} is 
deduced via specialization 
(Proposition~\ref{prop:mainC}). 
Section~\ref{sec:eAnr} discusses an 
$\fS_n$-equivariant analogue of the Eulerian 
polynomial $A_{n,r}(t)$. 

\section{Preliminaries}
\label{sec:pre}

This section provides key definitions, fixes notation 
and recalls a few useful facts about group representations 
and symmetric and quasisymmetric functions which will be 
used frequently. We will write $[a, b] = \{a, a+1,\dots,b\}$ 
for integers $a \le b$ and set $[n] := [1, n]$. 

\medskip
\noindent
\textbf{Colored permutations.}
The wreath product group $\ZZ_r \wr \fS_n$ consists of 
all permutations $w$ of $[n] \times [0, r-1]$
such that $w(a, 0) = (b, j) \Rightarrow w(a, i) = 
(b,i+j)$, where $i+j$ is computed modulo $r$ and the 
product of $\ZZ_r \wr \fS_n$ is composition of permutations. 
The elements of this group are represented as 
$r$-colored permutations of $[n]$, meaning pairs $(\sigma, 
\varepsilon)$ where $\sigma = (\sigma(1), 
\sigma(2),\dots,\sigma(n)) \in \fS_n$, $\varepsilon = 
(\varepsilon_1, \varepsilon_2,\dots,\varepsilon_n) \in 
[0,r-1]^n$ and $\varepsilon_k$ is thought of as the color 
assigned to $\sigma(k)$.

Given an $r$-colored permutation $w = (\sigma, 
\varepsilon)$, as above, we say that an index $k \in 
[n]$ is a \emph{descent} of $w$ if either $\varepsilon_k 
> \varepsilon_{k+1}$, or $\varepsilon_k = 
\varepsilon_{k+1}$ and $\sigma(k) > \sigma(k+1)$, where 
$\sigma(n+1) := n+1$ and $\varepsilon_{n+1} := 0$ (in 
particular, $n$ is a descent of $w$ if and only if 
$\sigma(n)$ has nonzero color). We denote by $\Des(w)$ 
the set of descents of $w$ and define the set of 
\emph{ascents} $\Asc(w)$ as the complement of $\Des(w)$ 
in $[n]$. An \emph{excedance} of $w$ is an index $k \in 
[n]$ such that either $\sigma(k) > k$, or $\sigma(k) = k$ 
and $\varepsilon_k$ is nonzero. The number of ascents, 
descents and excedances of $w$, respectively, will be 
denoted by $\asc(w)$, $\des(w)$ and $\exc(w)$. The 
\emph{flag excedance} number is defined as $\fexc(w) = 
r \cdot \exc_A(w) + \varepsilon_1 + \varepsilon_2 + \cdots 
+ \varepsilon_n$, where $\exc_A(w)$ is the number of 
indices $k \in [n]$ such that $\sigma(k) > k$ and 
$\varepsilon_k = 0$. More information, for instance 
about the generating functions of these statistics over 
$\ZZ \wr \fS_n$, and references can be found in 
\cite[Section~2]{Ath14}.

\medskip
\noindent
\textbf{Symmetric functions.} 
Our notation generally follows that in \cite{Mac95} 
\cite[Chapter~7]{StaEC2}. We denote by $\Lambda(x)$ the 
$\CC$-algebra of symmetric functions in a sequence 
$\bx = (x_1, x_2,\dots)$ of commuting independent 
indeterminates. We set 
\begin{eqnarray*}
E(\bx; z) & := & \sum_{n \ge 0} e_n(\bx) z^n \ = \ 
  \prod_{i \ge 1} \, (1 + x_i z), \\
H(\bx; z) & := & \sum_{n \ge 0} h_n(\bx) z^n \ = \ 
  \prod_{i \ge 1} \, \frac{1}{1 - x_i z},
\end{eqnarray*}
where	$e_n(\bx)$ and $h_n(\bx)$ are the elementary and 
complete homogeneous, respectively, symmetric functions 
of degree $n$ in $\bx$. We denote by 
$s_\lambda(\bx)$ the Schur function associated to a 
partition $\lambda$ and call a symmetric function 
\emph{Schur-positive}, if it can be written as a 
nonnegative integer linear combination of Schur 
functions.

Given a (complex, finite-dimensional) 
$\fS_n$-representation $\varphi$ with character $\chi$, 
the \emph{Frobenius characteristic} of $\varphi$ is 
defined by the formula
\[ \ch(\varphi)(\bx) \ = \ \frac{1}{n!} \, 
   \sum_{w \in \fS_n} \chi(w) p_w(\bx), \]
where $p_w(\bx) = p_\mu (\bx)$ for every permutation 
$w \in \fS_n$ of cycle type $\mu \vdash n$ and $p_\mu
(\bx)$ is a power sum symmetric function. We refer to
\cite[Section~7.18]{StaEC2} for a detailed discussion 
of the important properties of this map and mention 
that if $\varphi$ is non-virtual, then $\ch(\varphi)
(\bx)$ is a Schur-positive homogeneous symmetric 
function of degree $n$ which determines $\varphi$ up
to isomorphism. The map $\ch$ has a natural 
generalization to the wreath product group $\ZZ_r \wr 
\fS_n$ (see \cite[Section~I.B.6]{Mac95}), which we
denote here by $\ch_r$. The image under $\ch_r$ of 
any non-virtual $(\ZZ_r \wr \fS_n)$-representation 
$\varphi$ is \emph{Schur-positive} in $\Lambda(\bx^{(0)}) 
\otimes \Lambda (\bx^{(1)}) \otimes \cdots \otimes 
\Lambda(\bx^{(r-1)})$, meaning a nonnegative linear 
combination of products of Schur functions in the 
sequences of indeterminates $\bx^{(0)}, 
\bx^{(1)},\dots,\bx^{(r-1)}$ of total degree $n$, 
which determines $\varphi$ up to isomorphism. 

We recall that the \emph{fundamental quasisymmetric 
function} in $\bx$ of degree $n$, associated to $S 
\subseteq [n-1]$, is defined as
\[ F_{n, S} (\bx) \ = \ 
   \sum_{\scriptsize \begin{array}{c} i_1 \le i_2 
   \le \cdots \le i_n \\ j \in S \Rightarrow i_j 
   < i_{j+1} \end{array}} 
   x_{i_1} x_{i_2} \cdots x_{i_n}. \]
For a $\CC$-linear combination $f(\bx)$ of such 
functions, we will denote by ${\rm ex}^\ast (f)$ the 
evaluation of $f(1-q, (1-q)q, (1-q)q^2,\dots)$ at $q=1$ 
(this is a specialization of the $q$-analogue of the 
exponential specialization, discussed in 
\cite[Section~7.8]{StaEC2}). We then have ${\rm ex}^\ast 
(F_{n, S} (\bx)) = 1/n!$ for every $S \subseteq [n-1]$. 
This implies that 

\begin{itemize}
\itemsep=0pt

\item
${\rm ex}^\ast (e_n (\bx)) = {\rm ex}^\ast (h_n (\bx)) 
= 1/n!$ for every $n \in \NN$ and, more generally, 

\item
${\rm ex}^\ast (s_\lambda (\bx)) = f^\lambda/n!$ for 
every $\lambda \vdash n$, where $f^\lambda$ is the number
of standard Young tableaux of shape $\lambda$. 
\end{itemize}
As a result, ${\rm ex}^\ast (\ch(\varphi)(\bx))
= \dim(\varphi)/n!$ for every $\fS_n$-representation 
$\varphi$ and, by extending ${\rm ex}^\ast$ to formal 
power series with quasisymmetric functions as 
coefficients, ${\rm ex}^\ast (E(\bx; z)) = {\rm ex}^\ast 
(H(\bx; z)) = e^z$.   

\medskip
\noindent
\textbf{Polynomials and equivariant analogues.} 
We generally use the term `palindromicity', instead of 
`symmetry', for polynomials to avoid confusion with the 
notion of symmetric function. Thus, a polynomial $f(t) = 
\sum_j a_j t^j \in \NN[t]$ is called

\begin{itemize}
\item[$\bullet$] 
  \emph{palindromic}, with center (of symmetry) $n/2$, 
	if $a_j = a_{n-j}$ for all $j \in \ZZ$,
\item[$\bullet$] 
  \emph{unimodal}, if $a_0 \le a_1 \le \cdots \le a_k \ge 
	 a_{k+1} \ge \cdots$ for some $k \in \NN$,
\item[$\bullet$] 
  \emph{alternatingly increasing}, if $a_0 \le a_n \le 
	a_1 \le a_{n-1} \le \cdots \le a_{\lfloor (n+1)/2 
	\rfloor}$ for some $n \in \NN$ and $a_k = 0$ for $k>n$,
\item[$\bullet$] 
  \emph{$\gamma$-positive}, if 
	  \begin{equation} \label{eq:def-gamma}
      f(t) \ = \ \sum_{i=0}^{\lfloor n/2 \rfloor} \gamma_i 
			           t^i (1+t)^{n-2i} 
    \end{equation}
for some $n \in \NN$ and nonnegative integers $\gamma_0, 
\gamma_1,\dots,\gamma_{\lfloor n/2 \rfloor}$.
\end{itemize}

Gamma-positivity implies palindromicity and unimodality
and, in fact, it has developed into a powerful method 
to prove these properties; see \cite{Ath17} and
\cite[Chapter~4]{Pet15}. There are versions of this 
notion for nonpalindromic polynomials. We say that a 
nonzero polynomial $f(t) \in \NN[t]$ has \emph{center} 
$(a+b)/2$, where $a$ (respectively, $b$) is the smallest 
(respectively, largest) integer $k$ for which the 
coefficient of $t^k$ in $f(t)$ is nonzero. Then, $f(t)$ 
has a unique decomposition as a sum of two palindromic 
polynomials with centers of symmetry $n/2$ and $(n+1)/2$ 
(respectively, $n/2$ and $(n-1)/2$), where $n/2$ is the 
center of $f(t)$. Following \cite[Section~5.1]{Ath17}, we 
call $f(t)$ \emph{right $\gamma$-positive} (respectively, 
\emph{left $\gamma$-positive}) if these palindromic 
polynomials are $\gamma$-positive. Right (respectively, 
left) $\gamma$-positivity implies 
that $f(t)$ (respectively, $tf(t)$) is alternatingly 
increasing and, in particular, that $f(t)$ is unimodal.
For instance, the colored binomial Eulerian and colored 
derangement polynomials which appear in 
Theorems~\ref{thm:mainB} and~\ref{thm:Ath14} are right 
$\gamma$-positive and left $\gamma$-positive, respectively.

Given a finite group $G$,	a graded (non-virtual) 
$G$-representation $\varphi = \oplus_{j \ge 0} \, 
\varphi_j$ is called a \emph{$G$-equivariant analogue} 
of $f(t)$ if $\dim(\varphi_j) = a_j$ for every $j$. To 
extend the notions just introduced for polynomials to 
their $G$-equivariant analogues, it is convenient to 
consider the polynomial $\varphi(t) = \sum_j \varphi_j 
t^j$ instead and work with polynomials in $t$ whose 
coefficients are in the representation ring $R(G)$, i.e.,
formal integer linear combinations of (isomorphism 
classes of) irreducible $G$-representations. Then, one
simply replaces the usual total order on $\ZZ$ by the 
partial order $\le_G$ on $R(G)$ defined by setting 
$\varphi \le_G \psi$ if $\psi - \varphi$ is equal to 
a nonnegative integer linear combination of irreducible
$G$-representations (so that nonnegativity of integers 
is replaced by the property that elements of $R(G)$ 
are non-virtual $G$-representations). For instance, 
$\varphi(t)$ is \emph{$\gamma$-positive} if it can be 
expressed as in the right-hand side 
of~(\ref{eq:def-gamma}) for some $n \in \NN$ 
and non-virtual $G$-representations 
$\gamma_0, \gamma_1,\dots,\gamma_{\lfloor n/2 \rfloor}$.
For example, let $\varphi(t) = \varphi_0 + (\varphi_0 + 
\varphi_1)t + \varphi_0 t^2$, where $\varphi_0$ and 
$\varphi_1$ are nonisomorphic irreducible 
$G$-representations. Then, $\varphi(t)$ is palindromic 
and unimodal but not $\gamma$-positive, since $\varphi(t) = 
\gamma_0 (1+t)^2 + \gamma_1 t$ with $\gamma_0 = 
\varphi_0$ and $\gamma_1 = \varphi_1 - \varphi_0$.

For $G = \ZZ_r \wr \fS_n$, according to our previous 
discussion of $\ch_r$, the polynomial $\varphi(t) = 
\sum_j \varphi_j t^j \in R(G)[t]$ is $\gamma$-positive 
if and only if $\sum_j \ch_r(\varphi_j) t^j$ is \emph{Schur 
$\gamma$-positive}, meaning that
\[ \sum_j \ch_r(\varphi_j) t^j \ = \ \sum_{i=0}^{\lfloor 
   n/2 \rfloor} g_i \, t^i (1+t)^{n-2i} \]
for some $n \in \NN$ and Schur-positive functions 
$g_i \in \Lambda(\bx^{(0)}) \otimes \Lambda (\bx^{(1)}) 
\otimes \cdots \otimes \Lambda(\bx^{(r-1)})$. 

\section{Gamma-positivity of binomial Eulerian 
         polynomials}
\label{sec:mainB}

This section derives the $\gamma$-positivity 
statement of Theorem~\ref{thm:mainB} from a 
similar property of the derangement polynomials 
for $r$-colored permutations, proven 
in~\cite{Ath14}, and provides an interpretation 
for the corresponding $\gamma$-coefficients. The 
following lemma will be crucial.

\begin{lemma} \label{lem:fgh} 
For $n \in \NN$, let $f_n(t), g_n(t), h_n(t) \in 
\NN(t)$ be polynomials such that
\begin{eqnarray}
h_n(t) & = & \sum_{m=0}^n {n \choose m} t^{n-m}
g_m(t) \label{eq:gh} \\
& & \nonumber \\ 
g_m(t) & = & \sum_{k=0}^m {m \choose k} f_k(t)
\label{eq:fg}
\end{eqnarray}
for all $m, n \in \NN$. If 
\begin{equation}
\label{eq:fngamma}
  f_n(t) \ = \ \sum_{i=0}^{\lfloor n/2 \rfloor} 
  \xi_{n,i} \, t^i (1+t)^{n-2i}
\end{equation} 
for every $n \in \NN$, then 
\begin{equation}
\label{eq:hngamma}
  h_n(t) \ = \ \sum_{i=0}^{\lfloor n/2 \rfloor} 
  \left( \, \sum_{k=2i}^n {n \choose k} \xi_{k,i} 
  \right) t^i (1+t)^{n-2i}
\end{equation} 
for every $n \in \NN$. 

In particular, if $f_n(t)$ is palindromic (or 
$\gamma$-positive), with center $n/2$,
for every $n \in \NN$, then $h_n(t)$ has the same 
property for every $n \in \NN$.
\end{lemma}

\begin{proof}
Combining Equations~(\ref{eq:gh}) 
and~(\ref{eq:fg}) and changing the order of 
summation, we get
\begin{eqnarray*}
h_n(t) & = & \sum_{m=0}^n {n \choose m} t^{n-m}
\sum_{k=0}^m {m \choose k} f_k(t) \ = \ 
\sum_{k=0}^n f_k(t) \sum_{m=k}^n {n \choose m} 
{m \choose k} t^{n-m} \\
& & \\ 
& = & \sum_{k=0}^n f_k(t) \sum_{m=k}^n 
{n \choose k} {n-k \choose m-k} t^{n-m} \ = \ 
\sum_{k=0}^n {n \choose k} f_k(t) (1+t)^{n-k}. 
\end{eqnarray*}
Replacing $f_k(t)$ by the expression provided by 
Equation~(\ref{eq:fngamma}) and changing the 
order of summation again yields 
(\ref{eq:hngamma}) and the proof follows.
\end{proof}

The derangement polynomial for the group $\ZZ_r 
\wr \fS_n$, introduced and studied on this level 
of generality by Chow and Mansour 
\cite[Section~3]{CM10}, can be defined by either 
one of the formulas
\begin{eqnarray*}
d_{n,r}(t) & = & \sum_{w \in \dD_{n,r}} t^{\exc(w)} 
\\ & & \\ 
& = & \sum_{k=0}^n \, (-1)^{n-k} {n \choose k}
A_{k,r}(t),
\end{eqnarray*}
where $\dD_{n,r}$ is the set of derangements 
(colored permutations with no fixed point of zero 
color) in $\ZZ_r \wr \fS_n$. By the principle of 
inclusion-exclusion, we have
\begin{equation} \label{eq:Anrdnr}
A_{n,r}(t) \ = \ \sum_{k=0}^n {n \choose k} 
d_{k,r}(t)
\end{equation}
for all $n \in \NN$, where $d_{0,r}(t) := 1$. For 
later use (see Section~\ref{sec:localgal}), we also
recall the formula (\cite[Theorem~5~(iv)]{CM10}) 
\begin{equation} \label{eq:expdnr}
\sum_{n \ge 0} d_{n,r} (t) \, \frac{z^n}{n!} \ =
\ \frac{(1-t) e^{(r-1)tz}} {e^{rtz} - te^{rz}}. 
\end{equation}

The following statement is one of the main results 
of~\cite{Ath14}; it reduces to 
\cite[Theorem~2.13]{Ath17} for $r=1$.

\begin{theorem} \label{thm:Ath14} 
{\rm (\cite[Theorem~1.3]{Ath14})} 
For all positive integers $n,r$ we have
\begin{equation}
\label{eq:dnrsum}
d_{n,r}(t) \ = \ d^+_{n,r}(t) \, + \, d^-_{n,r}(t),
\end{equation}
where
\begin{eqnarray*}
d^+_{n,r}(t) & = & \sum_{i=1}^{\lfloor n/2 \rfloor} 
\xi^+_{n,r,i} \, t^i (1+t)^{n-2i} \\
& & \\ 
d^-_{n,r}(t) & = & \sum_{i=1}^{\lfloor (n+1)/2 
\rfloor} \xi^-_{n,r,i} \, t^i (1+t)^{n+1-2i} 
\end{eqnarray*}
and 

\begin{itemize}
\itemsep=0pt
\item[$\bullet$] 
$\xi^+_{n,r,i}$ is equal to the number of colored
permutations $w \in \ZZ_r \wr \fS_n$ for which 
$\Asc(w) \subseteq [2,n]$ has exactly $i$ elements, 
no two consecutive, and contains $n$, 
\item[$\bullet$] 
$\xi^-_{n,r,i}$ is equal to the number of colored
permutations $w \in \ZZ_r \wr \fS_n$ for which 
$\Asc(w) \subseteq [2,n-1]$ has exactly $i-1$ 
elements, no two consecutive.
\end{itemize}
\end{theorem}

Setting 
\begin{eqnarray}
A^+_{n,r}(t) & := & \sum_{k=0}^n {n \choose k} 
d^+_{k,r}(t) \label{eq:Adnr+} \\
& & \nonumber \\ 
A^-_{n,r}(t) & := & \sum_{k=0}^n {n \choose k} 
d^-_{k,r}(t), \label{eq:Adnr-} 
\end{eqnarray}
where $d^+_{0,r}(t) := 1$ and $d^-_{0,r}(t) := 0$, 
and in view of Equations~(\ref{eq:Anrdnr}) 
and~(\ref{eq:dnrsum}), we have
\begin{equation}
\label{eq:Anrsum}
A_{n,r}(t) \ = \ A^+_{n,r}(t) \, + \, A^-_{n,r}(t).
\end{equation}
Then (\ref{eq:binomAnrsum}) holds, with
\begin{eqnarray}
\widetilde{A}^+_{n,r}(t) & := & \sum_{m=0}^n 
{n \choose m} t^{n-m} A^+_{m,r}(t), 
\label{eq:defbinomAnr+} \\
& & \nonumber \\ 
\widetilde{A}^-_{n,r}(t) & := & \sum_{m=0}^n 
{n \choose m} t^{n-m} A^-_{m,r}(t). 
\label{eq:defbinomAnr-} 
\end{eqnarray}
As a result, the following corollary proves 
the first statement of Theorem~\ref{thm:mainB} 
(note also that it generalizes the interpretation 
for $\widetilde{\gamma}_{n,i}$ given in 
Theorem~\ref{thm:mainA}). 
\begin{corollary} \label{cor:mainB} 
For all positive integers $n,r$ with $r \ge 2$, 
we have
\begin{eqnarray*}
\widetilde{A}^+_{n,r}(t) & = &  
\sum_{i=0}^{\lfloor n/2 \rfloor} 
\widetilde{\gamma}^+_{n,r,i} \, t^i (1+t)^{n-2i} \\
& & \\ 
\widetilde{A}^-_{n,r}(t) & = &  
\sum_{i=1}^{\lfloor (n+1)/2 \rfloor} 
\widetilde{\gamma}^-_{n,r,i} \, t^i (1+t)^{n+1-2i}, 
\end{eqnarray*}
where
\begin{eqnarray}
\widetilde{\gamma}^+_{n,r,i} & := &  
\sum_{k=2i}^n {n \choose k} \xi^+_{k,r,i}  
\label{eq:gamma+formula} \\
& & \nonumber \\ 
\widetilde{\gamma}^-_{n,r,i} & := &  
\sum_{k=2i}^n {n \choose k} \xi^-_{k,r,i}.  
\label{eq:gamma-formula} 
\end{eqnarray}
As a result:

\begin{itemize}
\itemsep=0pt
\item[$\bullet$] 
$\widetilde{\gamma}^+_{n,r,i}$ is equal to the 
number of $w \in \ZZ_r \wr \fS_{n+1}$ for which 
$\Asc(w)$ has exactly $i+1$ elements, no two 
consecutive, and contains $n$ and $w(1) > w(2) > 
\cdots > w(m) = 1$ have all zero color, for some 
$m$,
\item[$\bullet$] 
$\widetilde{\gamma}^-_{n,r,i}$ is equal to the 
number of $w \in \ZZ_r \wr \fS_{n+1}$ for which 
$\Asc(w) \subseteq [n-1]$ has exactly $i$ elements, 
no two consecutive, and $w(1) > w(2) > \cdots > w(m) 
= 1$ have all zero color, for some $m$.
\end{itemize}
\end{corollary}

\begin{proof} 
The first part follows from 
Lemma~\ref{lem:fgh}, where the role of $f_n(t)$  
is played by $d^+_{n,r}(t)$ and $d^-_{n,r}(t)$, 
respectively, and Theorem~\ref{thm:Ath14}. The 
second follows by rewriting 
Equations~(\ref{eq:gamma+formula}) 
and~(\ref{eq:gamma-formula}) as 
\begin{eqnarray*}
\widetilde{\gamma}^+_{n,r,i} & = &  
\sum_{k=0}^{n-2i} {n \choose k} \xi^+_{n-k,r,i}  
\\ & & \\ 
\widetilde{\gamma}^-_{n,r,i} & = &  
\sum_{k=0}^{n-2i} {n \choose k} \xi^-_{n-k,r,i} 
\end{eqnarray*}
and using the interpretations for 
$\xi^+_{n,r,i}$ and $\xi^{-}_{n,r,i}$ provided by 
Theorem~\ref{thm:Ath14}.   
\end{proof}

\section{Triangulations}
\label{sec:trian}

This section recalls some definitions and constructions 
about simplicial complexes and their triangulations which 
are needed to complete the proof of Theorem~\ref{thm:mainB}. 
Familiarity with basic notions, such as the correspondence 
between abstract and geometric simplicial complexes, will 
be assumed (detailed expositions can be found in \cite{Bj95, 
DRS10, StaCCA}). All simplicial complexes considered here 
will be finite. We will denote by $\|K\|$ the polyhedron 
(union of all simplices) of a geometric simplicial complex 
$K$ and by $|V|$ and $2^V$ the cardinality and power set, 
respectively, of a finite set $V$. 

Consider two geometric simplicial complexes $K'$ and $K$ 
in some Euclidean space $\RR^N$, with corresponding 
abstract simplicial complexes $\Delta'$ and $\Delta$. 
Then, $K'$ is a triangulation of $K$, and $\Delta'$ 
is a triangulation of $\Delta$, if (a) every simplex of 
$K'$ is contained in some simplex of $K$; and (b) $\|K'\| 
= \|K\|$. Given a simplex $L \in K$ with 
corresponding face $F \in \Delta$, the triangulation 
$K'$ naturally restricts to a triangulation $K'_L$ of 
$L$. The subcomplex $\Delta'_F$ of $\Delta'$ 
corresponding to $K'_L$ is a triangulation of the 
abstract simplex $2^F$, called the \emph{restriction} 
of $\Delta'$ to $F$. We will call an abstract 
triangulation of a sphere (respectively, simplex) 
\emph{regular} if it can be realized geometrically as
the boundary complex (respectively, the complex of 
lower faces) of a simplicial polytope of one dimension 
higher.

A simplicial complex $K$ is called \emph{pure} 
if all its facets (faces which are maximal with 
respect to inclusion) have the same dimension 
and \emph{flag} if it contains every simplex whose 
one-dimensional skeleton is a subcomplex of $K$.
A pure simplicial complex $K$ is called 
\emph{shellable} if its facets can be linearly 
ordered, so that the intersection of any facet $L$, 
other than the first, with the union of all preceding 
ones is equal to the union of (one or more) codimension 
one faces of $L$. 

The enumerative invariants of simplicial 
complexes which will be of importance here 
are the $h$-polynomial of a triangulation 
of a sphere \cite[Chapter~II]{StaCCA} and 
the local $h$-polynomial of a triangulation 
of a simplex \cite{Sta92} 
\cite[Section~III.10]{StaCCA}. The 
\emph{$h$-polynomial} of an 
$(n-1)$-dimensional abstract simplicial 
complex $\Delta$ is defined as
\begin{equation} \label{eq:defh}
  h(\Delta, t) \ = \ \sum_{i=0}^n \, f_{i-1} 
  (\Delta) \, t^i (1-t)^{n-i},
\end{equation}
where $f_i (\Delta)$ is the number of 
$i$-dimensional faces of $\Delta$. Given a 
triangulation $\Gamma$ of an 
$(n-1)$-dimensional simplex $2^V$, the 
\emph{local $h$-polynomial} of $\Gamma$ (with
respect to $V$) is defined 
\cite[Definition~2.1]{Sta92} by the formula 
\begin{equation} \label{eq:deflocalh}
  \ell_V (\Gamma, t) \ = \sum_{F \subseteq V} 
  \, (-1)^{n - |F|} \, h (\Gamma_F, t),
\end{equation}
where $\Gamma_F$ is the restriction of $\Gamma$ 
to the face $F \in 2^V$. By the principle of 
inclusion-exclusion, we have 
\begin{equation} \label{eq:h-localh}
  h (\Gamma, t) \ = \sum_{F \subseteq V} 
  \ell_F (\Gamma_F, t).
\end{equation}

The polynomials (\ref{eq:defh}) and
(\ref{eq:deflocalh}) have especially attractive 
properties when $\Gamma$ and $\Delta$ triangulate 
an $(n-1)$-dimensional simplex or sphere, 
respectively \cite{Sta92} 
\cite[Chapter~III]{StaCCA}. For instance, they 
both have nonnegative and palindromic coefficients, 
with center $n/2$. Moreover, they 
are unimodal if $\Gamma$ and $\Delta$, respectively,
are regular triangulations and are conjectured to be 
$\gamma$-positive \cite[Conjecture~5.4]{Ath12} 
\cite[Conjecture~2.1.7]{Ga05} (see also 
\cite[Section~3]{Ath17}) when the triangulations 
are assumed to be flag.

\bigskip
\noindent
\textbf{The complex $\Delta(\Gamma)$.} 
Every triangulation of a simplex can be extended to 
a triangulation of a sphere of the same dimension in 
a way which preserves important properties, such as 
flagness and regularity. This construction, which we 
now recall, was exploited in~\cite{Ath12}.

Let $V = \{v_1, v_2,\dots,v_n\}$ be an $n$-element 
set and $\Gamma$ be a triangulation of the simplex 
$2^V$. Pick an $n$-element set $U = \{u_1, 
u_2,\dots,u_n\}$ which is disjoint from the vertex 
set of $\Gamma$ and denote by $\Delta(\Gamma)$ the 
collection of sets of the form $E \cup G$, where 
$E = \{ u_i: i \in I\}$ is a face of the simplex 
$2^U$ for some $I \subseteq [n]$ and $G$ is a face 
of the restriction $\Gamma_F$ of $\Gamma$ to the 
face $F = \{ v_i: i \in [n] \sm I\}$ of the simplex 
$2^V$ which is complementary to $E$. Clearly, 
$\Delta(\Gamma)$ is a simplicial complex which 
contains $2^U$ and $\Gamma$ as subcomplexes. When 
$\Gamma = 2^V$ is the trivial triangulation, the 
complex $\Delta(\Gamma)$ is combinatorially isomorphic  
to the boundary complex of the $n$-dimensional 
cross-polytope, defined as the convex hull of the 
set of unit coordinate vectors in $\RR^n$ and their 
negatives. Part of the following statement 
appeared in a more general setting 
in~\cite[Section~4]{Ath12}.
\begin{proposition} \label{prop:S(G)} 
The simplicial complex $\Delta(\Gamma)$ 
triangulates an $(n-1)$-dimensional sphere 
for every triangulation $\Gamma$ of the 
$(n-1)$-dimensional simplex $2^V$. Moreover:

\begin{itemize}
\itemsep=0pt
\item[{\rm (a)}]
if $\Gamma$ is a flag complex, then so is 
$\Delta(\Gamma)$,

\item[{\rm (b)}]
if $\Gamma$ is a regular triangulation, then so is 
$\Delta(\Gamma)$, 

\item[{\rm (c)}]
if all restrictions of $\Gamma$ 
to the faces of $2^V$ are shellable, then so is
$\Delta(\Gamma)$,

\item[{\rm (d)}]
\begin{equation} \label{eq:S(G)hpoly}
h(\Delta(\Gamma), t) \ = \ \sum_{F \subseteq V} 
t^{n-|F|} \, h(\Gamma_F, t).
\end{equation}
\end{itemize}
\end{proposition}

\begin{proof} 
The first statement holds because
$\Delta(\Gamma)$ naturally triangulates 
$\Delta(2^V)$. Part (a) follows from 
\cite[Proposition~4.6]{Ath12} (iii). For part 
(b), consider a geometric realization, say $K'$,
of $\Delta(\Gamma)$ in which the elements of $V$ 
are realized by the unit coordinate vectors in 
$\RR^n$ and those of $U$ by their negatives, and 
let $K$ be the subcomplex of $K'$ which realizes 
$\Gamma$. Since $\Gamma$ is regular, Lemma~4.3.5 
in~\cite{DRS10} guarantees that it can be 
extended to a regular triangulation of 
$\Delta(2^V)$, which is realized by the boundary 
complex of the $n$-dimensional 
cross-polytope. On the other hand, since the 
boundary of this polytope does not contain any 
line segment joining a unit coordinate vector 
to its negative, $K'$ is the only triangulation 
of its boundary complex extending $K$ and the 
proof follows. 

To prove part (c), assume that all restrictions of 
$\Gamma$ to the $2^n$ faces of $2^V$ are shellable. 
Then, the same holds for the restrictions of 
$\Delta(\Gamma)$ to the $2^n$ facets of $\Delta(2^V)$, 
since every such restriction is the simplicial 
join of the restriction of $\Gamma$ to a face $F \in 
2^V$ with the subsimplex of $2^U$ which is 
complementary to $F$ and taking the simplicial join
with a simplex obviously preserves shellability. We 
partition the set of facets 
of $\Delta(\Gamma)$ into $2^n$ blocks, according to
the facet of $\Delta(2^V)$ in which each facet of 
$\Delta(\Gamma)$ lies, and linearly order the $2^n$ 
facets of $\Delta(2^V)$ in any way which extends the 
inclusion order on their intersections with 
$U$. This gives a linear order on the $2^n$ blocks, 
which we may extend to a total order of 
all facets of $\Delta(\Gamma)$ by choosing a linear
order for the facets within each block which is a 
shelling order for the corresponding restriction of
$\Delta(\Gamma)$. We leave it to the reader to verify 
that this is indeed a shelling order for 
$\Delta(\Gamma)$. 

To prove 
part (d), we apply \cite[Equation~(4-2)]{Ath12} 
to the triangulation $\Delta(\Gamma)$ of 
$\Delta(2^V)$. Taking into account the fact 
that cones of triangulations of simplices have the
zero polynomial as their local $h$-polynomials 
\cite[p.~821]{Sta92}, we get
\[ h(\Delta(\Gamma), t) \ = \ \sum_{F \subseteq V} 
\ell_F (\Gamma_F, t) (1+t)^{n-|F|}. \]
Using the defining Equation~(\ref{eq:deflocalh}),
we conclude that
\begin{eqnarray*}
h(\Delta(\Gamma), t) & = & \sum_{F \subseteq V} 
\left( \, \sum_{G \subseteq F} \, (-1)^{|F \sm G|} 
\, h (\Gamma_G, t) \right) (1+t)^{n-|F|} \\ & & \\
& = & \sum_{G \subseteq V} h (\Gamma_G, x) \left(
\, \sum_{G \subseteq F \subseteq V} \, 
(-1)^{|F \sm G|} \, (1+t)^{n-|F|} \right) \\ 
& & \\ &=& \sum_{G \subseteq V} t^{n-|G|} \, 
h(\Gamma_G, t)
\end{eqnarray*}
and the proof follows.
\end{proof} 

\begin{example} \label{ex:h} \rm
Let $\Gamma_n$ be the (first) barycentric subdivision
of $2^V$, consisting of all chains of nonempty 
subsets of $V$. Then, $h(\Gamma_n,t) = A_n(t)$ (see,
for instance, \cite[Theorem~9.1]{Pet15}) and $\Gamma_n$ 
restricts to the barycentric subdivision of $2^F$ for 
every face $F \in 2^V$. Therefore, from 
Equations~(\ref{eq:S(G)hpoly}) 
and~(\ref{eq:def2binomAn}) we get 
\[ h(\Delta(\Gamma_n), t) \ = \ \sum_{k=0}^n 
   {n \choose k} t^{n-k} A_k(t) \ = \ 
	 \widetilde{A}_n(t). \]
It was shown in \cite[Section~10.4]{PRW08} that 
$\widetilde{A}_n(t) = h(\widetilde{\Delta}_n, t)$,
where $\widetilde{\Delta}_n$ is the boundary complex
of the $n$-dimensional simplicial stellohedron. This 
polytope can be constructed by successively stellarly
subdividing the boundary faces of an $n$-dimensional 
simplex which contain a fixed vertex, in any order 
of decreasing dimension of these faces. Although this 
may not be obvious, the complex $\widetilde{\Delta}_n$ 
is combinatorially isomorphic to $\Delta(\Gamma_n)$. 
We omit the proof of this fact, since it is not 
essential for the results of this paper (but mention, 
as a hint for the interested reader, that such an 
isomorphism maps the vertices of 
$\widetilde{\Delta}_n$ corresponding to the faces of 
the simplex which contain the fixed vertex to those 
of the barycentric subdivision of $2^V$ and the other 
vertices to the elements of $U$). The construction of 
$\Delta(\Gamma_n)$ will allow us in the following 
section to provide the right generalization of
$\widetilde{\Delta}_n$ to the setting of $r$-colored 
permutations. 
\end{example} 

\section{Edgewise subdivisions and the proof of 
Theorem~\ref{thm:mainB}}
\label{sec:be}

This section employs the construction of $\Delta
(\Gamma)$ when $\Gamma$ is the $r$-fold edgewise 
subdivision of the barycentric subdivision of 
a simplex to complete the proof of 
Theorem~\ref{thm:mainB}. 

We first recall the definition of edgewise subdivision. 
Let $\Delta$ be an abstract simplicial complex whose vertex 
set $V(\Delta) = \{v_1, v_2,\dots,v_m\}$ is equipped with a 
fixed linear order. Given a positive integer $r$, denote 
by $V_r(\Delta)$ the set of maps $f: V(\Delta) \to \NN$ 
such that $\supp(f) \in \Delta$ and $f(v_1) + f(v_2) + 
\cdots + f(v_m) = r$, where $\supp(f)$ is the set of all
$v \in V(\Delta)$ for which $f(v) \ne 0$. For $f 
\in V_r(\Delta)$, let $\iota(f): V(\Delta) \to \NN$ be the 
map defined by setting $\iota(f)(v_j) = f(v_1) + f(v_2) + 
\cdots + f(v_j)$ for $j \in [m]$. The \emph{$r$-fold 
edgewise subdivision} of $\Delta$ is the abstract simplicial 
complex $\esd_r(\Delta)$ on the vertex set $V_r(\Delta)$ of 
which a set $E \subseteq V_r(\Delta)$ is a face if the 
following two conditions are satisfied:

\begin{itemize}
\itemsep=0pt
\item[$\bullet$]
$\bigcup_{f \in E} \, \supp(f) \in \Delta$ and

\item[$\bullet$]
$\iota(f) - \iota(g) \in \{0, 1\}^{V(\Delta)}$, or $\iota(g) 
- \iota(f) \in \{0, 1\}^{V(\Delta)}$, for all $f, g \in E$.
\end{itemize}

\noindent
Clearly, $\esd_r(\Delta)$ is combinatorially isomorphic to 
$\Delta$ for $r=1$.

The simplicial complex $\esd_r(\Delta)$ can be realized 
as a triangulation of $\Delta$; see, for instance, 
\cite[Section~3.3.1]{Ath17} \cite[Section~6]{BR05} 
\cite{EG00} and references therein, where its importance 
and long history in mathematics is also discussed, and 
\cite[Section~4]{HPPS18}, where it is studied under the 
name \emph{canonical triangulation}. More precisely, if 
$K$ is a geometric realization of $\Delta$ with 
corresponding ordered vertex set $V(K)$, then 
$\esd_r(\Delta)$ can be realized 
by a triangulation of $K$ in which the vertex $f \in 
V_r(\Delta)$ is represented by the point in $\|K\|$ with 
barycentric coordinates $f(v_i)/r$ with respect to $V(K)$. 
The restriction of $\esd_r(\Delta)$ to $F \in \Delta$ 
coincides with the triangulation $\esd_r(2^F)$ of the 
simplex $2^F$ (where $F$ is considered with the induced 
linear order).

  \begin{figure}
  \epsfysize = 1.8 in 
  \centerline{\epsffile{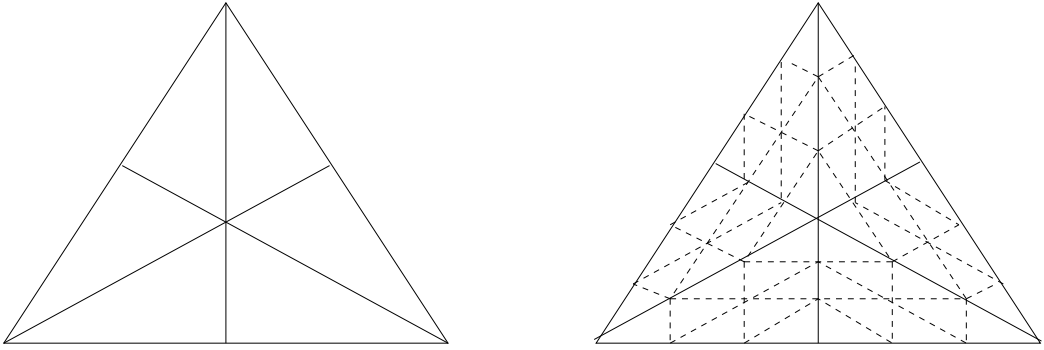}}
  \caption{The triangulations $\Gamma_3$ and $\Gamma_{3,3}$}
  \label{fig:KK33}
  \end{figure}
	
For the remainder of this section we set $\Gamma_{n,r} = 
\esd_r(\Gamma_n)$, where $\Gamma_n$ denotes the barycentric 
subdivision of the $(n-1)$-dimensional simplex (with its  
vertex set linearly ordered arbitrarily); see 
Figure~\ref{fig:KK33}. This triangulation was introduced in 
\cite{Ath14} in order to extend to the context of $r$-colored 
permutations Stanley's interpretation 
\cite[Proposition~2.4]{Sta92} of the derangement polynomial 
as a local $h$-polynomial. Since edgewise subdivision
preserves flagness and regularity (see, for instance, 
\cite[Theorem~4.11]{HPPS18}), $\Gamma_{n,r}$ is in fact a 
regular flag triangulation of the $(n-1)$-dimensional 
simplex. The relevance of $\Gamma_{n,r}$ to the study of 
Eulerian polynomials is explained by the following 
statement (recall that $A^+_{n,r}(t)$ was defined by 
Equation~(\ref{eq:Adnr+}) and the flag excedance number
$\fexc(w)$ for $w \in \ZZ_r \wr \fS_n$ was defined in 
Section~\ref{sec:pre}); part (c) is included for 
completeness.

\begin{proposition} \label{prop:Anr+} 
For all positive integers $n,r$:

\begin{itemize}
\itemsep=0pt
\item[{\rm (a)}]
$A^+_{n,r}(t) = h(\Gamma_{n,r}, t)$.

\item[{\rm (b)}]
\[ \frac{A^+_{n,r}(t)} {(1 - t)^n} \ = \ \sum_{k \ge 0}
   \left( (rk+1)^n - (rk)^n \right) t^k . \]

\item[{\rm (c)}]
\[ A^+_{n,r}(t) \ = \sum_{w \in (\ZZ_r \wr \fS_n)^+} 
 t^{\des(w)} \ = \sum_{w \in (\ZZ_r \wr \fS_n)^b} 
                   t^{\fexc(w)/r}, \]
where $(\ZZ_r \wr \fS_n)^+$ is the set of colored 
permutations $w \in \ZZ_r \wr \fS_n$ with first 
coordinate of zero color and $(\ZZ_r \wr \fS_n)^b$ 
is the set of $w \in \ZZ_r \wr \fS_n$ such that the 
sum of the colors of the coordinates of $w$ is 
divisible by $r$. 
\end{itemize}
\end{proposition}

\begin{proof}  
It was shown in \cite[Section~5]{Ath14}
that $\ell_V(\Gamma_{n,r}, t) = d^+_{n,r} (t)$. 
Thus, part (a) follows from this fact and 
Equations~(\ref{eq:Adnr+}) and~(\ref{eq:h-localh}). 
It was also shown in \cite[Section~5]{Ath14} 
that 
\[ h(\Gamma_{n,r}, t) \ = \ {\rm E}_r \left( 
   (1 + t + t^2 + \cdots + t^{r-1})^n \, A_n (t) 
   \right) \ = \sum_{w \in (\ZZ_r \wr \fS_n)^b} 
                   t^{\fexc(w)/r}, \]
where ${\rm E}_r: \RR[t] \to \RR[t]$ is the 
linear operator defined by setting ${\rm E}_r 
(t^k) = t^{k/r}$, if $k$ is divisible by $r$, 
and ${\rm E}_r (t^k) = 0$ otherwise. Furthermore,
the equality 
\[ \sum_{w \in (\ZZ_r \wr \fS_n)^+} x^{\des(w)} 
   \ = \ {\rm E}_r \left( (1 + t + t^2 + \cdots + 
   t^{r-1})^n \, A_n (t) \right) \]
follows by setting $q=1$ in the Carlitz identity 
\cite[Theorem~A.1]{BB07} for $\ZZ_r \wr \fS_n$. 
Thus, part (c) follows from part (a) and the 
previous remarks. For part (b), we use the previous 
formulas and Worpitzky's identity $A_n(t)/(1-t)^{n+1} 
= \sum_{k \ge 1} k^n t^{k-1}$ to conclude that 

\begin{eqnarray*}
A^+_{n,r}(t) & = & 
{\rm E}_r \left( (1 + t + t^2 + \cdots + t^{r-1})^n 
\, A_n (t) \right) \ = \ 
{\rm E}_r \left( \, \frac{(1-t^r)^n}{(1 - t)^n} \, 
A_n (t) \right) \\ & & \\
& = & (1-t)^n \, {\rm E}_r \left( \, \frac{A_n(t)}
{(1 - t)^n} \right) \ = \ (1-t)^n \, {\rm E}_r 
\left( \, \sum_{k \ge 1} k^n (t^{k-1} - t^k) \right) 
\\ & & \\ &=& (1-t)^n \, \sum_{k \ge 0} 
\left( (rk+1)^n - (rk)^n \right) t^k
\end{eqnarray*}
and the proof follows.
\end{proof} 

\begin{remark} \label{rem:referee} \rm
Let $P = \{ (x_1, x_2,\dots,x_n) \in \RR^n: 0 \le x_i 
\le r\}$ be the $r$th dilate of the standard unit
$n$-dimensional cube. Then, $(rk+1)^n - (rk)^n$ 
is equal to the 
number of lattice points in the $k$th dilate of the 
union of the $n$ facets of $P$ which do not contain 
the origin. Thus, part (b) of Proposition~\ref{prop:Anr+} 
shows that $A^+_{n,r}(t)$ can be interpreted as the 
$h^\ast$-polynomial (defined as in 
Section~\ref{sec:stap}) of a lattice polyhedral complex, 
namely the collection of all faces of the $n$ facets of 
$P$ which do not contain the origin.
\qed
\end{remark}

The following statement completes the proof of 
Theorem~\ref{thm:mainB}.

\begin{corollary} \label{cor:mainB-hpoly} 
We have $\widetilde{A}^+_{n,r}(t) = 
h(\Delta(\Gamma_{n,r}), t)$ for all positive 
integers $n,r$. In particular, 
$\widetilde{A}^+_{n,r}(t)$ is equal to the 
$h$-polynomial of the boundary complex of 
an $n$-dimensional flag simplicial polytope.
\end{corollary}

\begin{proof} 
The first statement is an immediate consequence
of Proposition~\ref{prop:Anr+} (a) and 
Equations~(\ref{eq:defbinomAnr+}) 
and~(\ref{eq:S(G)hpoly}). In view of 
Proposition~\ref{prop:S(G)}, the second statement 
follows from the first and the fact that $\Gamma_{n,r}$ 
is a regular flag triangulation. 
\end{proof} 

\medskip
Interpretations for the polynomials $d^+_{n,r}(t)$ 
and $d^{-}_{n,r}(t)$, which appear in 
Theorem~\ref{thm:Ath14}, as the local 
$h^\ast$-polynomials of certain $s$-lecture hall 
simplices were found by Gustafsson and 
Solus~\cite{GS18}. This raises the following 
question.

\begin{question}
Is there an Ehrhart-theoretic interpretation of 
$\widetilde{A}^{-}_{n,r} (t)$, possibly similar to the 
ones provided for $d^+_{n,r}(t)$ and $d^{-}_{n,r} (t)$ 
in \cite{GS18}? 
\end{question}

\section{Group actions}
\label{sec:act}

This section provides background from \cite{Ath18} 
\cite[Section~4]{Sta92} \cite{Ste94} which will be 
necessary to study $\fS_n$-equivariant analogues of 
colored Eulerian, derangement and binomial Eulerian
polynomials in the following sections. We assume
familiarity with Stanley--Reisner theory, Ehrhart theory 
and the homology of posets and refer the reader to the 
sources \cite{BR15, HiAC, Sta82, StaCCA, Wa07} for 
detailed expositions. 

\subsection{Fans and simplicial complexes}
\label{sec:fans}
Our discussion here follows closely 
\cite[Section~1]{Ste94} and \cite[Section~4]{Sta92}.
A \emph{simplicial fan} in Euclidean space $\RR^n$ 
is defined as a finite collection $\fF$ of pointed 
simplicial cones in $\RR^n$ with apex at the origin,
such that (a) every face of every cone in $\fF$ belongs 
to $\fF$; and (b) the intersection of any two cones in 
$\fF$ is a face of both. To such a fan $\fF$ one can 
associate an abstract simplicial complex $\Delta$ on 
the vertex set of one-dimensional cones (rays) of $\fF$ 
in the obvious way. We denote by $\|\fF\|$ the union 
of the cones of $\fF$ and assume that $\Delta$ is 
homeomorphic to the $(n-1)$-dimensional sphere (this 
happens if $\|\fF\| = \RR^n$, in which case $\fF$ is 
called \emph{complete}), or to the $(n-1)$-dimensional 
ball; in particular, $\Delta$ is Cohen--Macaulay over 
$\CC$ (and any other field).   

Suppose that $G$ is a finite group of orthogonal
transformations of $\RR^n$ which acts on $\fF$ and 
denote by $V(\Delta)$ the vertex set of $\Delta$. Then, 
$G$ acts simplicially on $\Delta$ and linearly on the
polynomial ring $S = \CC[x_v: v \in V(\Delta)]$ in 
commuting indeterminates which are in one-to-one 
correspondence with the vertices of $\Delta$. This 
action preserves the ideal $I_\Delta$ of $S$ which 
is generated by the square-free monomials which 
correspond to the non-faces of $\Delta$ and hence 
$G$ acts linearly on the \emph{face ring} $\CC[\Delta] 
= S / I_\Delta$ as well. As explained on 
\cite[p.~250]{Ste94}, the group $G$ leaves invariant 
the $\CC$-linear span of the linear forms 
\begin{equation} \label{eq:lsop}
   \theta_i \ = \ \sum_{v \in V(\Delta)} \langle 
   v, \ree_i \rangle x_v 
\end{equation}
for $i \in [n]$, where $\langle \ , \, \rangle$ 
is the standard inner product on $\RR^n$ and $(\ree_1,
\ree_2,\dots,\ree_n)$ is the basis of unit coordinate
vectors (any other basis works, but this 
one will be convenient in the sequel). As a result,
$G$ acts linearly on each homogeneous component of the 
(standard) graded ring 
\[ \CC(\Delta) \ := \ \CC[\Delta] / \Theta \ = \ 
   \bigoplus_{j=0}^n \CC(\Delta)_j, \]
where $\Theta$ is the ideal of $\CC[\Delta]$ generated 
by $\theta_1,\theta_2,\dots,\theta_n$. Since (by 
\cite[Lemma~III.2.4]{StaCCA}) this sequence is a 
linear system of parameters for $\CC[\Delta]$, which is 
Cohen--Macaulay over $\CC$, we have 
$\sum_{j=0}^n \dim(\CC(\Delta)_j) t^j = h(\Delta,t)$ and 
hence $\CC(\Delta) =  \oplus_{j=0}^n \CC(\Delta)_j$ is 
a $G$-equivariant analogue of $h(\Delta,t)$. By a 
theorem of Danilov~\cite{Da78}, if $\fF$ is complete 
and its cones are generated by elements of some lattice, 
then the space $\CC(\Delta)_j$ is isomorphic to the 
cohomology of degree $2j$ over $\CC$ of the complex 
projective toric variety associated to $\fF$ (although 
this is not the point of view adopted here). The pair 
$(\Delta, G)$ is said to satisfy the \emph{equivariant 
Gal phenomenon} \cite[Section~5]{SW17+} if the polynomial 
$\sum_{j=0}^n \CC(\Delta)_j t^j \in R(G)[t]$ is 
$\gamma$-positive (as defined towards the end of 
Section~\ref{sec:pre}). 

Following \cite[Section~1]{Ste94}, we say that the action 
of $G$ on $\Delta$ is \emph{proper} (and that $\fF$ 
carries a proper $G$-action) if every $w \in G$ fixes 
pointwise the vertices of every face $F \in \Delta$ which 
is fixed by $w$. Under this assumption, the set 
$\Delta^w$ of faces of $\Delta$ which are fixed by $w$ 
forms an induced subcomplex of $\Delta$, for every 
$w \in G$. Assuming that $G = \fS_n$ acts on 
$\RR^n$ by permuting coordinates and combining 
Theorem~1.4 with the considerations in Section~6 of 
\cite{Ste94}, we conclude that the formula
\begin{equation} \label{eq:ste}
  \sum_{j=0}^n \ch (\CC(\Delta)_j) (\bx) t^j \ = \ 
  \frac{1}{n!} \sum_{w \in \fS_n} 
  \frac{h(\Delta^w, t)}{(1 - t)^{1+\dim(\Delta^w)}}
  \, \prod_{i \ge 1} \, (1 - t^{\lambda_i(w)}) \,
  p_{\lambda_i(w)} (\bx)
\end{equation}
holds, provided $\fS_n$ acts properly on $\Delta$, 
where $\lambda_1(w) \ge \lambda_2(w) \ge \cdots$ are 
the lengths of the cycles of $w \in \fS_n$ and $p_k(\bx)$ 
is a power sum symmetric function. Stembridge 
\cite[Section~6]{Ste94} used this formula to give a 
new proof of Equation~(\ref{eq:genMn}), essentially due 
to Procesi~\cite{Pro90}, where $\varphi_n$ is the graded 
$\fS_n$-representation $\CC(\Delta_n) = \bigoplus_{j=0}^n 
\CC(\Delta_n)_j$ obtained from the $\fS_n$-action on 
the Coxeter fan associated to $\fS_n$ (we note that,
although \cite[Theorem~1.4]{Ste94} is stated for complete
fans, its proof exploits the Cohen--Macaulayness of 
$\Delta$ and hence applies in our situation).

Suppose now that $\|\fF\|$ is the positive orthant in 
$\RR^n$, generated by the unit coordinate vectors 
$\ree_1, \ree_2,\dots,\ree_n$, which is thus 
simplicially subdivided by $\fF$. We will denote by 
$\Gamma$ the associated simplicial complex, which can 
be considered as a triangulation of the geometric simplex 
$\Sigma_n$ on the vertex set $V_n = \{\ree_1, 
\ree_2,\dots,\ree_n\}$. Recall that $G$ acts on $\fF$ 
by orthogonal transformations and note that, in this 
case, $G$ must be a subgroup of the automorphism group 
$\fS_n$ of $\Sigma_n$. The \emph{local face module} 
\[ L_{V_n} (\Gamma) \ = \ \bigoplus_{j=0}^n 
   L_{V_n} (\Gamma)_j \]
is a graded $\CC$-vector space (and $S$-module) 
defined \cite[Definition~4.5]{Sta92} as the image in 
$\CC(\Gamma)$ of the ideal of $\CC[\Gamma]$ generated 
by the square-free monomials which correspond to the 
faces of $\Gamma$ lying in the relative interior of 
$\Sigma_n$. Since the coefficient of $x_v$ in the 
right-hand side of Equation~(\ref{eq:lsop}) is equal 
to zero for every $v \in V(\Gamma)$ lying on the facet
of $\Sigma_n$ opposite to $\ree_i$, the sequence 
$\theta_1,\theta_2,\dots,\theta_n$ is a special linear 
system of parameters for $\CC[\Gamma]$, in the sense of 
\cite[Definition~4.2]{Sta92}. By \cite[Theorem~4.6]{Sta92}, 
we have $\sum_{j=0}^n \dim (L_{V_n}(\Gamma))_j t^j = 
\ell_{V_n} (\Gamma, t)$. As explained in 
\cite[p.~823]{Sta92}, the group $G$ acts on each 
homogeneous component of $L_{V_n}(\Gamma)$, which 
becomes a graded $G$-representation and a $G$-equivariant 
analogue of $\ell_{V_n} (\Gamma, t)$. The pair 
$(\Gamma, G)$ satisfies the \emph{local equivariant Gal 
phenomenon} \cite[Section~5]{Ath17} \cite[Section~5]{Ath18} 
if the polynomial $\sum_{j=0}^n L_{V_n}(\Gamma)_j t^j \in
R(G)[t]$ is $\gamma$-positive.

Our prototypical example is the standard action of 
$\fS_n$, viewed as the group of symmetries of the 
simplex $\Sigma_n$, on the barycentric subdivision of 
$\Sigma_n$ (this action is easily verified to be proper). 
The following result combines \cite[Proposition~4.20]{Sta92} 
with an identity due to Gessel \cite[Equation~(87)]{Ath17} 
\cite[Equation~(6.3)]{SW17+} (a proof of which is 
given in \cite[Section~4]{Ath18}).
\begin{proposition} \label{prop:esdn}
For the $\fS_n$-action on the barycentric subdivision 
$\Gamma_n$ of the simplex $\Sigma_n$ and the 
corresponding graded $\fS_n$-representation $\psi_n = 
\oplus_{j=0}^n \, \psi_{n,j}$ on the local face module 
$L_{V_n}(\Gamma_n)$, we have
\[ 1 \, + \, \sum_{n \ge 1} z^n \, \sum_{j=0}^n 
  \ch (\psi_{n,j}) (\bx) t^j \ = \ 
  \frac{1 - t}{H(\bx; tz) - tH(\bx; z)}. \]
Moreover, the polynomial $\sum_{j=0}^n \psi_{n,j} t^j$
is $\gamma$-positive for every $n \ge 1$.
\end{proposition}

A careful examination of the proof of 
\cite[Proposition~4.20]{Sta92} shows that it actually 
yields the following more general statement. By the 
standard embedding $\RR^m \hookrightarrow \RR^n$, we 
may consider $\Sigma_m$ as a face of $\Sigma_n$ for 
$m \le n$.

\begin{proposition} \label{prop:staH}
Let $K_n$ be a triangulation of the simplex $\Sigma_n$, 
for every $n \ge 1$, so that $K_m$ is the restriction of 
$K_n$ to $\Sigma_m$ for all $m \le n$. Suppose that for 
every $n \ge 1$, the full automorphism group $\fS_n$ of 
$\Sigma_n$ acts on $K_n$. Then,

\[ 1 \, + \, \sum_{n \ge 1} z^n \, \sum_{j=0}^n \ch 
   (\CC(K_n)_j) (\bx) t^j \ = \ H(\bx; z) \left( 
	 1 \, + \, \sum_{n \ge 1} z^n \, \sum_{j=0}^n \ch 
	 (L_{V_n}(K_n)_j) (\bx) t^j \right). \]
\end{proposition}

\medskip
\subsection{Posets}
Group actions on posets induce representations on their 
homology which often have combinatorial significance; 
see \cite{Sta82} \cite[Chapter~2]{Wa07}. Here we describe 
a special situation which will be useful in 
the following sections, referring to \cite{Ath18} for 
more details and to \cite[Chapter~3]{StaEC1} for any 
undefined poset terminology.

Given finite graded posets $\pP$ and $\qQ$ with rank 
functions $\rho_\pP$ and $\rho_\qQ$, respectively, 
the \emph{Rees product} $\pP \ast \qQ$ was defined by 
Bj\"orner and Welker~\cite{BjW05} as the set $\{ (p, q) 
\in \pP \times \qQ: \rho_\pP (p) \ge \rho_\qQ (q) \}$, 
partially ordered by setting $(p_1, q_1) \preceq (p_2, 
q_2)$ if the following conditions are satisfied:

\begin{itemize}
\itemsep=0pt
\item[$\bullet$]
$p_1 \preceq p_2$ holds in $\pP$,

\item[$\bullet$]
$q_1 \preceq q_2$ holds in $\qQ$ and

\item[$\bullet$]
$\rho_\pP (p_2) - \rho_\pP (p_1) \ge \rho_\qQ (q_2) 
- \rho_\qQ (q_1)$.
\end{itemize}

\noindent
The poset $\pP \ast \qQ$ is graded and if a finite group 
$G$ acts on $\pP$ by order-preserving bijections, then 
it does so on $\pP \ast \qQ$ as well by acting trivially 
on the second coordinate of its elements. Thus, there is 
an induced $G$-representation on the homology
$\widetilde{H}_\ast (\pP \ast \qQ; \CC)$. 

Assuming $\pP$ is bounded, let us denote by $\pP^-$, 
$\pP_{-}$ and $\bar{\pP}$ the poset obtained from $\pP$ 
by removing its maximum element, or minimum element, or both,
respectively, and by $T_{t,n}$ the poset whose Hasse 
diagram is a complete $t$-ary tree of height $n$, rooted 
at the minimum element. The following statement is a 
direct consequence of \cite[Theorem~1.2]{Ath18}, which is 
an equivariant analogue of \cite[Corollary~3.8]{LSW12}; it
plays a key role in the proofs of all equivariant 
$\gamma$-positivity results in the following sections.
\begin{theorem} {\rm (\cite{Ath18})} \label{thm:Athrees}
Let $\pP$ be a finite bounded poset of rank $n+1$ which
is Cohen--Macaulay over $\CC$ and $G$ be a finite group 
which acts on $\pP$ by order-preserving bijections. Then, 
there exist non-virtual $G$-representations 
$\beta^+_{\pP,i}$, $\beta^{-}_{\pP,i}$, 
$\gamma^+_{\pP,i}$ and $\gamma^{-}_{\pP,i}$ such that 

\begin{equation} \label{eq:Ath1} 
\widetilde{H}_{n-1} (\bar{\pP} \ast T_{t,n-1}; \CC) 
\ \cong_G \ \sum_{i=1}^{\lfloor n/2 \rfloor} 
\beta^+_{\pP,i} \, t^i (1+t)^{n-2i} \ + 
\sum_{i=0}^{\lfloor (n-1)/2 \rfloor} \beta^{-}_{\pP,i} 
\, t^i (1+t)^{n-1-2i}
\end{equation}
and
\begin{equation} \label{eq:Ath2}
\widetilde{H}_{n-1} ((\pP^- \ast T_{t,n})_{-}; \CC)  
\ \cong_G \ \sum_{i=0}^{\lfloor n/2 \rfloor}
\gamma^+_{\pP,i} \, t^i (1+t)^{n-2i} \ + 
\sum_{i=1}^{\lfloor (n+1)/2 \rfloor} 
\gamma^{-}_{\pP,i} \, t^i (1+t)^{n+1-2i}
\end{equation}
for every positive integer $t$.
\end{theorem}

\medskip
\subsection{Lattice points}
\label{sec:stap}
Let $P$ be an $n$-dimensional lattice polytope in 
$\RR^N$. The \emph{$h^\ast$-polynomial} (or 
\emph{Ehrhart $h$-polynomial}) of $P$ is defined 
by the equation
\[ \sum_{k \ge 0} |kP \cap \ZZ^N| \, t^k \ = \ 
   \frac{h^\ast(P,t)}{(1-t)^{n+1}}, \]
where $kP$ is the $k$th dilate of $P$. The 
function $h^\ast(P,t)$ is indeed a polynomial in $t$, 
with nonnegative integer coefficients and degree not 
exceeding $n$, which is very well studied in algebraic
and geometric combinatorics; see \cite[Chapter~3]{BR15}
\cite[Section~4.6]{StaEC1} and references therein. The 
problem to investigate when $h^\ast(P,t)$ is unimodal, in 
particular, has been of great interest; see \cite{Br16}
for a recent survey on this topic.

Equivariant analogues of $h^\ast(P,t)$ are provided by 
Stapledon's equivariant Ehrhart theory~\cite{Stap11} as 
follows. Suppose $G$ is a finite group which acts linearly 
on $\RR^N$, preserving the lattice $\ZZ^N$, and leaves $P$ 
invariant. Assume further that the affine span of $P$ 
contains the origin and let $\lL$ be its intersection 
with $\ZZ^N$. Stapledon~\cite{Stap11} defines the formal 
power series $\varphi^\ast_P(t)$, with coefficients in the 
representation ring of $G$, via the generating function 
formula
\begin{equation} \label{eq:def-phi*}
  \sum_{k \ge 0} \chi_{kP} \, t^k \ = \ 
  \frac{\varphi^\ast_P (t)}{(1-t) \det(I - \rho t)}, 
\end{equation}
where $\chi_{kP}$ is the permutation representation
defined by the $G$-action on the set of lattice points 
of $kP$ and $\rho: G \to GL(\lL)$ is the induced 
representation.  

The series $\varphi^\ast_P(t)$ is a 
$G$-equivariant analogue of $h^\ast(P,t)$ which, under
additional assumptions (see \cite[Section~7]{Stap11}) is 
a polynomial in $t$ whose coefficients are non-virtual 
$G$-representations. For example, if $P$ is the standard
unit cube in $\RR^n$ on which the symmetric group $\fS_n$ 
acts by permuting coordinates, then $h^\ast(P,t) = A_n(t)$ 
and $\varphi^\ast_P(t) = \sum_{j=0}^{n-1} \varphi_{n,j} 
t^j$, in the notation of Section~\ref{sec:intro}. For 
this particular $\fS_n$-action on an $n$-dimensional 
lattice polytope $P$ in $\RR^n$, and identifying 
$G$-representations with their characters,
Equation~(\ref{eq:def-phi*}) yields that 
\begin{equation} \label{eq:fSn-phi*}
  \sum_{k \ge 0} \chi_{kP}(w) t^k \ = \ 
  \frac{\varphi^\ast_P (t)(w)} {(1-t) \prod_{i \ge 1}
  (1 - t^{\lambda_i(w)})}
\end{equation}
for every $w \in \fS_n$, where $\lambda_i(w)$ are the 
lengths of the cycles of $w$.

\section{Equivariant analogues of colored derangement
         polynomials}
\label{sec:localgal}

This section studies an $\fS_n$-equivariant analogue 
of the colored derangement polynomial $d_{n,r}(t)$ 
and confirms the local equivariant Gal phenomenon for
the natural $\fS_n$-action on the triangulation 
$\Gamma_{n,r}$, considered in Section~\ref{sec:be}. 
These results generalize Proposition~\ref{prop:esdn} 
(the case $r=1$) and are partially extended in the 
following section, which focuses on the colored binomial 
Eulerian polynomials.  

We define (up to isomorphism) the graded 
$\fS_n$-representation $\psi_{n,r} = \oplus_{j=0}^n 
\, \psi_{n,r,j}$ by the generating function formula
\begin{equation} \label{eq:defpsi-nr}
  1 \, + \, \sum_{n \ge 1} z^n \, \sum_{j=0}^n 
  \ch (\psi_{n,r,j}) (\bx) t^j \ = \ 
	\frac{(1-t) H(\bx; tz)^{r-1}} 
	{H(\bx; tz)^r - tH(\bx; z)^r}.
\end{equation}
Applying the specialization ${\rm ex}^\ast$ (see 
Section~\ref{sec:pre}) on both sides gives
\[ 1 \, + \, \sum_{n \ge 1} \, \frac{z^n}{n!} \, 
\sum_{j=0}^n \dim (\psi_{n,r,j}) t^j \ = \ 
\frac{(1-t) e^{(r-1)tz}} {e^{rtz} - te^{rz}} \]
and hence, a comparison with Equation~(\ref{eq:expdnr})
shows that $\psi_{n,r}$ is indeed an $\fS_n$-equivariant 
analogue of $d_{n,r}(t)$. Clearly, for $r=1$ it reduces 
to the graded $\fS_n$-representation $\psi_n$ which 
appears in Proposition~\ref{prop:esdn}.

To define precisely the $\fS_n$-action on $\Gamma_{n,r}$
that was mentioned earlier, consider an abstract 
simplicial complex $\Delta$, as in the beginning of 
Section~\ref{sec:be}, on a linearly ordered vertex set 
$V(\Delta)$. Suppose that $G$ acts simplicially on 
$\Delta$ and preserves the induced linear order on the 
vertex sets of faces of $\Delta$, in other words, if 
$\{u, v\} \in \Delta$ and $u$ precedes $v$, then $w \cdot 
u$ precedes $w \cdot v$ for every $w \in G$ (note that 
such an action is automatically proper). Then, $G$ acts 
on $V_r(\Delta)$ by the rule $(w \cdot f)(v) = f(w^{-1} 
\cdot v)$ for $w \in G$, $f \in V_r(\Delta)$ and
$v \in V(\Delta)$. The proof of the following statement
is fairly straightforward.

\begin{proposition} \label{prop:edgeact}
Under the stated assumptions:

\begin{itemize}
\itemsep=0pt

\item [(a)]
the $G$-action on $V_r(\Delta)$ induces a proper 
simplicial $G$-action on $\esd_r(\Delta)$, and 

\item [(b)]
the subcomplex $\esd_r(\Delta)^w$ is combinatorially 
isomorphic to $\esd_r(\Delta^w)$ for every $w \in G$, 
where the vertex set of $\Delta^w$ is considered with 
the induced linear order. 
\end{itemize}
\end{proposition}
\begin{proof}
For part (a), given a face $E \in \esd_r(\Delta)$ and 
$w \in G$, we need to show that $w(E) := \{ w \cdot f: 
f \in E\} \in \esd_r(\Delta)$. Since, by definition of 
the $G$-action on $V_r(\Delta)$, $\supp(w \cdot f) = w 
\cdot \supp(f)$ for every $f \in V_r(\Delta)$, we have 
\[ \bigcup_{f \in E} \, \supp(w \cdot f) = w \left(
\bigcup_{f \in E} \, \supp(f) \right). \]
Thus, the first condition in the definition of 
$\esd_r(\Delta)$ that the union of the supports of the 
elements of $E$ is a face of $\Delta$ transfers to $w(E)$.
The second condition transfers as well because the 
$G$-action respects the given linear order on $V(\Delta)$
and hence the nonzero values of $\iota(w \cdot f) - 
\iota(w \cdot g)$ are exactly those of 
$\iota(f) - \iota(g)$. To verify that the $G$-action on
$\esd_r(\Delta)$ is proper, suppose that $w \in G$ fixes 
$E \in \esd_r(\Delta)$. Then, $w$ fixes the union, say $F 
\in \Delta$, of the supports of all elements of $E$. Since, 
as has already been commented, the $G$-action on $\Delta$ 
is proper, $w$ fixes $F$ pointwise. Therefore, it fixes all 
elements of $V_r(\Delta)$ whose support is contained in
$F$; in particular, $w$ fixes $E$ pointwise. 

Part (b) follows easily from the fact that the $G$-actions 
on $\Delta$ and $\esd_r(\Delta)$ are proper; the details 
are left to the reader.
\end{proof}

Following the notation of Section~\ref{sec:act}, we
linearly order the vertex set of the barycentric 
subdivision $\Gamma_n$ of the simplex $\Sigma_n$ in 
any way which respects the inclusion order on the 
corresponding faces of $\Sigma_n$ and define the 
$r$-fold edgewise subdivision $\Gamma_{n,r}$ using 
this order. Since the $\fS_n$-action on $\Gamma_n$ 
preserves the inclusion order on its faces, it induces 
a proper $\fS_n$-action on $\Gamma_{n,r}$ which, as 
discussed in Section~\ref{sec:be}, can be realized as 
a triangulation of $\Sigma_n$ which refines $\Gamma_n$. 
Thus, $\fS_n$ acts linearly on the local face module 
$L_{V_n}(\Gamma_{n,r})$ and the resulting graded 
$\fS_n$-representation is an $\fS_n$-equivariant 
analogue of the local $h$-polynomial $\ell_{V_n} 
(\Gamma_{n,r},t)$. The latter was shown in 
\cite{Ath14} to equal the polynomial $d_{n,r}^+(t)$ 
which appeared in Theorem~\ref{thm:Ath14}. Therefore, 
the following statement may be viewed as an 
$\fS_n$-equivariant analogue of part of 
\cite[Theorems 1.2 and~1.3]{Ath14}.

\begin{theorem} \label{thm:psi-nr}
There exists an isomorphism 
\begin{equation} \label{eq:isopsi-nr} 
\psi_{n,r} \, \cong_{\fS_n} \, \psi_{n,r}^+ \, \oplus \, 
\psi_{n,r}^{-} 
\end{equation}
of graded $\fS_n$-representations, where $\psi_{n,r}^+ 
= \oplus_{j=0}^n \, \psi_{n,r,j}^+$ and $\psi_{n,r}^{-} 
= \oplus_{j=0}^n \, \psi_{n,r,j}^{-}$ are such that

\begin{itemize}
\itemsep=0pt

\item
$\sum_{j=0}^n \, \psi_{n,r,j}^+ t^j$ is a
$\gamma$-positive polynomial with center $n/2$ and 
zero constant term, 

\item
$\sum_{j=0}^n \, \psi_{n,r,j}^{-} t^j$ is a
$\gamma$-positive polynomial with center $(n+1)/2$ and 
zero constant term. 
\end{itemize}

Moreover, $\psi_{n,r}^+$ is isomorphic to the graded
$\fS_n$-representation on the local face module $L_{V_n}
(\Gamma_{n,r})$, induced by the $\fS_n$-action on
$\Gamma_{n,r}$, and
\begin{eqnarray} \label{eq:genpsi+nr}
1 \, + \, \sum_{n \ge 1} z^n \, \sum_{j=0}^n \ch 
(\psi_{n,r,j}^+) (\bx) t^j & = & 
\frac{H(\bx; tz)^{r-1} - tH(\bx; z)^{r-1}} 
     {H(\bx; tz)^r - tH(\bx; z)^r}
\\ & & \nonumber \\ 
\sum_{n \ge 1} z^n \, \sum_{j=0}^n \ch 
(\psi_{n,r,j}^{-}) (\bx) t^j & = & 
\frac{t(H(\bx; z)^{r-1} - H(\bx; tz)^{r-1})} 
     {H(\bx; tz)^r - tH(\bx; z)^r}. 
		 \label{eq:genpsi-nr}
\end{eqnarray}
\end{theorem}

\begin{remark} \rm
For $r=2$, it was shown in \cite[Proposition~5.2]{Ath18} 
that Equation~(\ref{eq:genpsi+nr}) holds for the graded 
$\fS_n$-representation on the local face module of a 
triangulation of the $(n-1)$-dimensional simplex 
(sometimes referred to as the interval triangulation) 
which is different from $\Gamma_{n,2}$. This should come 
as no surprise, in view of the similarities in the 
combinatorial properties of the two triangulations; see 
\cite[Remark~4.5]{Ath16a}. 
\qed
\end{remark}

We first confirm the statement of 
Theorem~\ref{thm:psi-nr} about $L_{V_n}(\Gamma_{n,r})$ 
using the machinery of \cite{Ste94}, explained in 
Section~\ref{sec:act}.
\begin{proposition} \label{prop:eGamma-nr}
The graded $\fS_n$-representation on the local face 
module $L_{V_n}(\Gamma_{n,r})$, induced by the 
$\fS_n$-action on the triangulation $\Gamma_{n,r}$, 
satisfies Equation~(\ref{eq:genpsi+nr}). 
\end{proposition}
\begin{proof} 
Since, for any fixed $r$, Proposition~\ref{prop:staH} 
applies to the family of triangulations $\Gamma_{n,r}$, 
it suffices to show that
\begin{equation} \label{eq:eGnr}
  1 \, + \, \sum_{n \ge 1} z^n \, \sum_{j=0}^n 
  \ch (\CC(\Gamma_{n,r})_j) (\bx) t^j \ = \ 
  \frac{H(\bx; z) H(\bx; tz)^{r-1} - tH(\bx; z)^r}
	{H(\bx; tz)^r - tH(\bx; z)^r}.
\end{equation}
This can be achieved by a computation similar to those 
in \cite[Section~5]{Ath18} \cite[Section~6]{Ste94} as 
follows. We will apply Equation~(\ref{eq:ste}) to the 
proper $\fS_n$-action on $\Delta := \Gamma_{n,r}$. Since, 
as the reader can verify, $\Gamma^w$ is combinatorially 
isomorphic to $\Gamma_{c(w)}$, where $c(w)$ is the number 
of cycles of $w \in \fS_n$, part (b) of 
Proposition~\ref{prop:edgeact} shows that 
$\Gamma_{n,r}^w$ is combinatorially isomorphic to 
$\Gamma_{c(w),r}$. As a result, and in 
view of Proposition~\ref{prop:Anr+} (a), we get 
\begin{eqnarray} \nonumber
  \sum_{j=0}^n \ch (\CC(\Gamma_{n,r})_j) (\bx) t^j z^n
  & = & \frac{1}{n!} \sum_{w \in \fS_n} 
	\frac{A^+_{c(w),r}(t)}{(1 - t)^{c(w)}} \, 
	\prod_{i \ge 1} \, (1 - t^{\lambda_i(w)}) \, 
	p_{\lambda_i(w)} (\bx) z^{\lambda_i(w)} \\
	& & \nonumber \\ 
	& = & \sum_{\lambda = (\lambda_1, \lambda_2,\dots) \,
  \vdash \, n} m^{-1}_\lambda \, 
  \frac{A^+_{\ell(\lambda),r}(t)}
  {(1 - t)^{\ell(\lambda)}} \, \prod_{i \ge 1} \, 
  (1 - t^{\lambda_i}) \, p_{\lambda_i} (\bx)
  z^{\lambda_i}, \label{eq:eGnr-first}
\end{eqnarray}
where $n!/m_\lambda$ is the cardinality of the 
conjugacy class of $\fS_n$ which corresponds to 
$\lambda \vdash n$ and $\ell(\lambda)$ is the number 
of parts of $\lambda$. By Proposition~\ref{prop:Anr+} 
(b), the expression (\ref{eq:eGnr-first}) may be 
rewritten as
\[ \sum_{k \ge 0} \, t^k 
   \sum_{\lambda = (\lambda_1, \lambda_2,\dots) \,
   \vdash \, n} m^{-1}_\lambda 
   \left( (rk+1)^{\ell(\lambda)} - 
   (rk)^{\ell(\lambda)} \right) \, \prod_{i \ge 1} \, 
   (1 - t^{\lambda_i}) \, p_{\lambda_i} (\bx)
   z^{\lambda_i} . \]
Summing over all $n \ge 1$ and using the identity 
(see the proof of \cite[Proposition~3.3]{Ste92})
\[ \sum_{\lambda = (\lambda_1, \lambda_2,\dots)} 
   m^{-1}_\lambda \, k^{\ell(\lambda)} \, 
	 \prod_{i \ge 1} \, (1 - t^{\lambda_i}) \, 
	 p_{\lambda_i} (\bx) z^{\lambda_i} \ = \ 
	 \frac{H(\bx; z)^k}{H(\bx; tz)^k} \] 
we get
\begin{eqnarray*}
  1 \, + \, \sum_{n \ge 1} z^n \, \sum_{j=0}^n 
  \ch (\CC(\Gamma_{n,r})_j) (\bx) t^j & = & 1 \, + \, 
  \sum_{k \ge 0} \, t^k 
  \left( \frac{H(\bx; z)^{rk+1}}{H(\bx; tz)^{rk+1}} 
  - \frac{H(\bx; z)^{rk}}{H(\bx; tz)^{rk}} \right) 
  \\ & & \\
  & = & 1 \, + \, \left( \frac{H(\bx; z)}{H(\bx; tz)} 
  - 1 \right) \left( 1 - t \, \frac{H(\bx; z)^r}
  {H(\bx; tz)^r} \right)^{-1} \\ & & \\
  & = & \frac{H(\bx; z) H(\bx; tz)^{r-1} - tH(\bx; z)^r} 
        {H(\bx; tz)^r - tH(\bx; z)^r}
\end{eqnarray*}
and the proof of~(\ref{eq:eGnr}) follows.
\end{proof}

To proceed with the proof of Theorem~\ref{thm:psi-nr},
just as in \cite[Section~3]{Ath14}, we need to consider 
the colored analogue $\bB_{n,r}$ of the Boolean lattice 
of subsets of $[n]$. An \emph{$r$-colored} subset of 
$[n]$ is any subset of $[n] \times \ZZ_r$ 
which contains at most one pair $(i, j)$ 
for each $i \in [n]$ (where $j$ is thought of as the 
color assigned to $i$). The set $\bB_{n,r}$ consists of 
all $r$-colored subsets of $[n]$ and is partially ordered 
by inclusion. The group $\ZZ_r \wr \fS_n$ acts on
$r$-colored subsets of $[n]$ by permuting their elements 
and cyclically shifting their colors. This action induces 
an action on the poset $\bB_{n,r}$ by order-preserving 
bijections. We also recall from Section~\ref{sec:pre} the 
notation $\ch_r$ for the Frobenius characteristic for 
$\ZZ_r \wr \fS_n$ and that the image under this map of 
any non-virtual representation of $\ZZ_r \wr \fS_n$ is 
Schur-positive in $\Lambda(\bx^{(0)}) \otimes \Lambda 
(\bx^{(1)}) \otimes \cdots \otimes \Lambda(\bx^{(r-1)})$, 
of total degree $n$. 

\medskip
\noindent
\emph{Proof of Theorem~\ref{thm:psi-nr}.} Our plan is 
to apply the first part of Theorem~\ref{thm:Athrees} 
to the poset $\pP$ obtained from $\bB_{n,r}$ 
by adding a maximum element and the group $G = \ZZ_r 
\wr \fS_n$, which acts on $\pP$ by order-preserving 
bijections. We first claim that 

\[ 1 \, + \, \sum_{n \ge 1} \, \ch_r  \left( 
\widetilde{H}_{n-1} ((\bB_{n,r} \sm \{\varnothing\}) 
\ast T_{t,n-1}; \CC) \right) z^n \ = \ 
\frac{(1-t) \prod_{i=1}^{r-1} E(\bx^{(i)}; z)} 
{\prod_{i=0}^{r-1} E(\bx^{(i)}; tz) - t 
 \prod_{i=0}^{r-1} E(\bx^{(i)}; z)} \]

\medskip
\noindent
and omit the proof, which is entirely similar to that
of the special case $r=2$, treated in 
\cite[Proposition~4.3]{Ath18}, and uses properties of 
$\ch_r$ which are direct generalizations of those of
$\ch_2$ mentioned in \cite[Section~2]{Ath18}. We now 
replace the 
representation on the left-hand side of this formula 
with the right-hand side of the isomorphism~(\ref{eq:Ath1}), 
set all sequences $\bx^{(0)}, \bx^{(1)},\dots,\bx^{(r-1)}$ 
equal to each other and apply the standard involution 
$\omega$. Since these operations on the functions 
$\ch_r (\beta^+_{\pP,i})$ and $\ch_r(\beta^{-}_{\pP,i})$ 
preserve their Schur-positivity, we conclude that 
\begin{eqnarray*} 
\frac{(1-t) H(\bx; z)^{r-1}}{H(\bx; tz)^r - tH(\bx; z)^r} 
& = & 1 \, + \, \sum_{n \ge 1} z^n \,
\sum_{i=1}^{\lfloor n/2 \rfloor} \xi^+_{n,r,i} (\bx) 
\, t^i (1+t)^{n-2i} \\ & & \\ & & \ \ \, + \,  
\sum_{n \ge 1} z^n \, \sum_{i=0}^{\lfloor (n-1)/2 \rfloor} 
\xi^{-}_{n,r,i} (\bx) \, t^i (1+t)^{n-1-2i} 
\end{eqnarray*}	
for some Schur-positive symmetric functions $\xi^+_{n,r,i} 
(\bx)$ and $\xi^{-}_{n,r,i} (\bx)$. Replacing $z$ with 
$tz$ and $t$ with $1/t$, this equation may be rewritten as
\begin{eqnarray} \label{eq:genxi-nr+-} 
\frac{(1-t) H(\bx; tz)^{r-1}}{H(\bx; tz)^r - tH(\bx; z)^r} 
& = & 1 \, + \, \sum_{n \ge 1} z^n \,
\sum_{i=1}^{\lfloor n/2 \rfloor} \xi^+_{n,r,i} (\bx) 
\, t^i (1+t)^{n-2i} \\ & & \nonumber \\ & & \ \ \, + \,  
\sum_{n \ge 1} z^n \, \sum_{i=0}^{\lfloor (n-1)/2 \rfloor} 
\xi^{-}_{n,r,i} (\bx) \, t^{i+1} (1+t)^{n-1-2i}. \nonumber
\end{eqnarray}	

We define the graded $\fS_n$-representations $\psi_{n,r}^+$ 
and $\psi_{n,r}^{-}$ by Equations~(\ref{eq:genpsi+nr}) 
and~(\ref{eq:genpsi-nr}) and note that (\ref{eq:isopsi-nr}) 
holds, since the right-hand side, say $F(\bx,t; z)$, of 
(\ref{eq:defpsi-nr}) equals the sum of the right-hand 
sides, say $F^+(\bx,t; z)$ and $F^{-}(\bx,t; z)$, of 
(\ref{eq:genpsi+nr}) and~(\ref{eq:genpsi-nr}), 
respectively. 

Finally, we observe that $F^+(\bx,t; z)$ is left invariant
under replacing $z$ with $tz$ and $t$ with $1/t$. This 
implies that the coefficient of $z^n$ in $F^+(\bx,t; z)$ 
is a palindromic polynomial in $t$, centered at $n/2$. 
Similarly, we find that the coefficient of $z^n$ in 
$F^{-}(\bx,t; z)$) is a palindromic polynomial in $t$,
centered at $(n+1)/2$. Since the corresponding 
properties are clear for the coefficient of $z^n$ in the 
two summands in the right-hand side of 
Equation~(\ref{eq:genxi-nr+-}) and because of the 
uniqueness of the decomposition of a polynomial $f(t)$ 
as a sum of two palindromic polynomials with centers 
$n/2$ and $(n+1)/2$, we must have 
\begin{eqnarray*} 
\sum_{j=0}^n \ch (\psi_{n,r,j}^+) (\bx) t^j & = & 
\sum_{i=1}^{\lfloor n/2 \rfloor} \xi^+_{n,r,i} (\bx) 
\, t^i (1+t)^{n-2i} , \\ & & \\ 
\sum_{j=0}^n \ch (\psi_{n,r,j}^{-}) (\bx) t^j & = & 
\sum_{i=0}^{\lfloor (n-1)/2 \rfloor} \xi^{-}_{n,r,i} 
(\bx) \, t^{i+1} (1+t)^{n-1-2i}
\end{eqnarray*}	
and the proof follows. 
\qed

\begin{problem}
Find explicit combinatorial interpretations of the 
coefficients of the expansions of $\xi^+_{n,r,i} (\bx)$ 
and $\xi^{-}_{n,r,i} (\bx)$ as linear combinations of 
Schur functions.
\end{problem}

\section{Equivariant analogues of colored binomial 
         Eulerian polynomials}
\label{sec:gal}

This section studies an $\fS_n$-equivariant analogue 
of $\widetilde{A}_{n,r}(t)$. For $r=2$, it confirms the 
equivariant Gal phenomenon for the $\fS_n$-action on the 
triangulated sphere $\Delta(\Gamma_{n,r})$, considered in 
Section~\ref{sec:be}, and deduces new combinatorial 
interpretations of the $\gamma$-coefficients of 
$\widetilde{A}^+_{n,r} (t)$ and $\widetilde{A}^{-}_{n,r} 
(t)$.

\subsection{Equivariant analogue of 
$\widetilde{A}_{n,r}(t)$}

Following the approach 
of Section~\ref{sec:localgal}, we define the graded 
$\fS_n$-representation $\widetilde{\varphi}_{n,r} = 
\oplus_{j=0}^n \, \widetilde{\varphi}_{n,r,j}$ by the 
formula
\begin{equation} \label{eq:deftphi-nr}
  1 \, + \, \sum_{n \ge 1} z^n \, \sum_{j=0}^n 
  \ch (\widetilde{\varphi}_{n,r,j}) (\bx) t^j \ = \ 
	\frac{(1-t) H(\bx; z) H(\bx; tz)^r} 
	{H(\bx; tz)^r - tH(\bx; z)^r}.
\end{equation}
Applying the specialization ${\rm ex}^\ast$ on both
sides gives
\[ 1 \, + \, \sum_{n \ge 1} \, \frac{z^n}{n!} \, 
\sum_{j=0}^n \dim (\widetilde{\varphi}_{n,r,j}) t^j \ = \ 
\frac{(1-t) e^{(rt+1)z}} {e^{rtz} - te^{rz}} \]
which, in view of Equation~(\ref{eq:binomAnrexpgen}) 
(see the proof of Proposition~\ref{prop:mainC}), shows 
that $\widetilde{\varphi}_{n,r}$ is 
an $\fS_n$-equivariant analogue of 
$\widetilde{A}_{n,r}(t)$. For $r=1$, it reduces to the 
graded $\fS_n$-representation $\widetilde{\varphi}_n$ 
discussed in Section~\ref{sec:intro}.

As in Section~\ref{sec:localgal}, we consider
$\Gamma_{n,r}$ as a triangulation of the simplex $\Sigma_n$
on which $\fS_n$ acts by permuting coordinates. This action 
extends to one on $\Delta(\Gamma_{n,r})$ which, as explained 
in Section~\ref{sec:be}, can be considered as a triangulation
of the boundary complex of the $n$-dimensional 
cross-polytope (note that 
this extended $\fS_n$-action is not proper, since 
$-\Sigma_n$ is a face of $\Delta(\Gamma_{n,r})$ which is 
fixed by the action, although not pointwise). As a result, 
we have a linear $\fS_n$-action on 
$\CC(\Delta(\Gamma_{n,r}))$. The following theorem, in 
view of our discussion in Example~\ref{ex:h}, partially 
generalizes \cite[Theorem~5.2]{SW17+} (the case $r=1$); 
we have no reason to doubt that the $\gamma$-positivity 
statement holds for every $r \ge 2$.

\begin{theorem} \label{thm:tphi-nr}
There exists an isomorphism 
\begin{equation} \label{eq:isotphi-nr} 
\widetilde{\varphi}_{n,r} \, \cong_{\fS_n} \, 
\widetilde{\varphi}_{n,r}^+ \, \oplus \, 
\widetilde{\varphi}_{n,r}^{-} 
\end{equation}
of graded $\fS_n$-representations, where 
$\widetilde{\varphi}_{n,r}^+ = \oplus_{j=0}^n \, 
\widetilde{\varphi}_{n,r,j}^+$ and 
$\widetilde{\varphi}_{n,r}^{-} = \oplus_{j=0}^n \, 
\widetilde{\varphi}_{n,r,j}^{-}$ are such that

\begin{itemize}
\itemsep=0pt

\item
$\sum_{j=0}^n \, \widetilde{\varphi}_{n,r,j}^+ t^j$ is 
palindromic and unimodal, with center $n/2$, and 

\item
$\sum_{j=0}^n \, \widetilde{\varphi}_{n,r,j}^{-} t^j$ 
is palindromic with center $(n+1)/2$ and zero 
constant term. 
\end{itemize}

Moreover, these polynomials are $\gamma$-positive for
$r=2$, we have 

\begin{eqnarray} \label{eq:gentphi+nr}
1 \, + \, \sum_{n \ge 1} z^n \, \sum_{j=0}^n \ch 
(\widetilde{\varphi}_{n,r,j}^+) (\bx) t^j & = & 
\frac{H(\bx; z) H(\bx; tz)^r - tH(\bx; z)^r H(\bx; tz)} 
     {H(\bx; tz)^r - tH(\bx; z)^r}
\\ & & \nonumber \\ 
\sum_{n \ge 1} z^n \, \sum_{j=0}^n \ch 
(\widetilde{\varphi}_{n,r,j}^{-}) (\bx) t^j & = & 
  \frac{t H(\bx; z) H(\bx; tz) 
  (H(\bx; z)^{r-1} - H(\bx; tz)^{r-1})} 
     {H(\bx; tz)^r - tH(\bx; z)^r} 
		 \label{eq:gentphi-nr}
\end{eqnarray}

\medskip
\noindent
for $r \ge 1$ and 
$\widetilde{\varphi}_{n,r}^+$ is 
isomorphic to the graded $\fS_n$-representation 
$\CC(\Delta(\Gamma_{n,r}))$, induced by the 
$\fS_n$-action on $\Delta(\Gamma_{n,r})$.
\end{theorem}

The proof of Theorem~\ref{thm:tphi-nr} follows a similar 
path with that of Theorem~\ref{thm:psi-nr}. One essential
difference is that the $\fS_n$-action on 
$\Delta(\Gamma_{n,r})$ is not proper and hence one cannot 
hope to adapt the proof of Proposition~\ref{prop:eGamma-nr}. 
The following proposition, whose proof we postpone for a 
while, will be used instead.

\begin{proposition} \label{prop:athaH}
Let $(K_n)$ be a sequence of triangulations, as in 
Proposition~\ref{prop:staH}. Assuming that each $K_n$ is 
shellable, we have

\[ 1 \, + \, \sum_{n \ge 1} z^n \, \sum_{j=0}^n \ch 
   (\CC(\Delta(K_n))_j) (\bx) t^j \ = \ H(\bx; tz) \left( 
   1 \, + \, \sum_{n \ge 1} z^n \, \sum_{j=0}^n \ch 
	 (\CC(K_n)_j) (\bx) t^j \right). \]
\end{proposition}

\bigskip
\noindent
\emph{Proof of Theorem~\ref{thm:tphi-nr}.} 
We define the graded $\fS_n$-representations 
$\widetilde{\varphi}_{n,r}^+$ and 
$\widetilde{\varphi}_{n,r}^{-}$ by 
Equations~(\ref{eq:gentphi+nr}) 
and~(\ref{eq:gentphi-nr}). As in the proof of 
Theorem~\ref{thm:psi-nr} we find that (\ref{eq:isotphi-nr}) 
holds and that the polynomials $\sum_{j=0}^n \, 
\widetilde{\varphi}_{n,r,j}^+ t^j$ and $\sum_{j=0}^n \, 
\widetilde{\varphi}_{n,r,j}^{-} t^j$ are symmetric, with 
center of symmetry $n/2$ and $(n+1)/2$, respectively, 
for every $n$. The last statement of the theorem follows 
from Proposition~\ref{prop:athaH}, applied to the 
triangulation $\Gamma_{n,r}$ (which is regular, and hence 
shellable), and Equation~(\ref{eq:eGnr}) and implies the 
unimodality of $\sum_{j=0}^n \, 
\widetilde{\varphi}_{n,r,j}^+ t^j$ via an application of 
the hard Lefschetz theorem, as pointed out by Stanley on 
\cite[p.~528]{Sta89}. 

To prove the $\gamma$-positivity statement for $r=2$, 
we apply the second part of Theorem~\ref{thm:Athrees} 
to the $(\ZZ_2 \wr \fS_n)$-action on the poset $\pP$ 
obtained from $\bB_{n,2}$ by adding a maximum element. 
The left-hand side of Equation~(\ref{eq:Ath2}) in this 
case has already been computed in 
\cite[Proposition~4.6]{Ath18}; its image under $\ch_2$
was shown to equal the coefficient of $z^n$ in 

\begin{equation} \label{eq:ereesfn}
  \frac{(1-t) E(\by; z) E(\bx; tz) E(\by; tz)} 
  {E(\bx; tz)E(\by; tz) - tE(\bx; z)E(\by; z)}, 
\end{equation}

\medskip
\noindent
where $\bx := \bx^{(0)}$ and $\by := \bx^{(1)}$. Setting 
$\bx = \by$, applying the involution $\omega$ and 
comparing with Equation~(\ref{eq:deftphi-nr}), we get
\[ \sum_{j=0}^n \ch (\widetilde{\varphi}_{n,2,j}) (\bx) 
   t^j \ = \ \sum_{i=0}^{\lfloor n/2 \rfloor} 
   \widetilde{\gamma}^+_{n,2,i} (\bx) \, t^i (1+t)^{n-2i} 
   \ + \sum_{i=1}^{\lfloor (n+1)/2 \rfloor} 
   \widetilde{\gamma}^{-}_{n,2,i} (\bx) \, t^i 
   (1+t)^{n+1-2i} \]
for some Schur-positive symmetric functions 
$\widetilde{\gamma}^+_{n,2,i} (\bx)$ and 
$\widetilde{\gamma}^{-}_{n,2,i} (\bx)$. Using the 
uniqueness of the palindromic decomposition 
(\ref{eq:isotphi-nr}), as in the proof of 
Theorem~\ref{thm:psi-nr} we conclude
that 
\begin{eqnarray*} 
\sum_{j=0}^n \ch (\widetilde{\varphi}_{n,2,j}^+) (\bx) 
t^j & = & \sum_{i=0}^{\lfloor n/2 \rfloor} 
\widetilde{\gamma}^+_{n,2,i} (\bx) \, t^i (1+t)^{n-2i} , 
\\ & & \\ 
\sum_{j=0}^n \ch (\widetilde{\varphi}_{n,2,j}^{-}) (\bx) 
t^j & = & \sum_{i=1}^{\lfloor (n+1)/2 \rfloor} 
\widetilde{\gamma}^{-}_{n,2,i} (\bx) \, t^i 
(1+t)^{n+1-2i} 
\end{eqnarray*}	
and the proof follows. 
\qed

\begin{problem}
Find explicit combinatorial interpretations for the 
coefficients of the expansions of 
$\widetilde{\gamma}^+_{n,2,i} (\bx)$ and 
$\widetilde{\gamma}^{-}_{n,2,i} (\bx)$ as linear
combinations of Schur functions.
\end{problem}

\begin{remark} \label{rem:r>2} \rm
One may try to extend this proof to any $r \ge 2$ by 
considering the $(\ZZ_r \wr \fS_n)$-action on $\bB_{n,r}$,
just as in the proof of Theorem~\ref{thm:psi-nr}. However,
the image of the left-hand side of Equation~(\ref{eq:Ath2}) 
under $\ch_r$ turns out to be the coefficient of $z^n$ in 

\[ \frac{(1-t) \prod_{i=1}^{r-1} E(\bx^{(i)}; z)
   \prod_{i=0}^{r-1} E(\bx^{(i)}; tz)}  
   {\prod_{i=0}^{r-1} E(\bx^{(i)}; tz) - t 
   \prod_{i=0}^{r-1} E(\bx^{(i)}; z)}. \]

\medskip
\noindent
After specializing and applying $\omega$ this expression
becomes 

\[ \frac{(1-t) H(\bx; z)^{r-1} H(\bx; tz)^r} 
	   {H(\bx; tz)^r - tH(\bx; z)^r}, \]

\medskip
\noindent
which equals the right-hand side of 
Equation~(\ref{eq:deftphi-nr}) only for $r=2$. 
\qed
\end{remark}

For the proof of Proposition~\ref{prop:athaH}, which
is somewhat technical, we assume familiarity with the 
notion of shellability and its connections to 
Stanley--Reisner theory, as explained in 
\cite[Section~III.2]{StaCCA}. The assumption that each
$K_n$ is shellable simplifies the proof, but we make 
no claim that it is necessary for the validity of the 
proposition. 

\medskip
\noindent
\emph{Proof of Proposition~\ref{prop:athaH}.} 
As was the case with $\Delta(\Gamma_{n,r})$, we consider 
$\Delta(K_n)$ as a triangulation of the boundary complex
$\Delta(\Sigma_n)$ of the $n$-dimensional cross-polytope,
defined as the convex hull of the set of unit coordinate
vectors in $\RR^n$ and their negatives. 

Part (c) of Proposition~\ref{prop:S(G)} applies to $K_n$
and hence, we may choose a shelling order of the set of  
facets of $\Delta(K_n)$ as in the proof given there. 
As described in \cite[Theorem~2.5]{StaCCA}, this
shelling order gives rise to a basis $\bB$ of the 
$\CC$-vector space $\CC(\Delta(K_n))$ which consists
of monomials, one for every facet of $\Delta(K_n)$.
We recall that the faces of $\Sigma_n$ are in 
one-to-one correspondence with the facets of 
$\Delta(\Sigma_n)$. For a face $F$ of $\Sigma_n$, we 
will denote by $\bB_F$ the set of those basis elements 
which are associated to the facets of the restriction 
of $\Delta(K_n)$ to the facet of $\Delta(\Sigma_n)$ 
corresponding to $F$ and by $W_F$ the $\CC$-linear span 
of $\bB_F$ in $\CC(\Delta(K_n))$. Then, each $W_F$ is a 
(standard) graded $\CC$-vector space and we have a decomposition
\begin{equation} \label{eq:direct}
  \CC(\Delta(K_n)) \ = \ \bigoplus_F \, W_F \ = \ 
  \bigoplus_{m=0}^n \, W_{n,m} 
\end{equation}
of graded $\CC$-vector spaces, where $F$ ranges over 
all faces of $\Sigma_n$ and $W_{n,m}$ is the direct sum 
of those spaces $W_F$ for which $F$ has codimension $m$.

We claim that $W_F$ is isomorphic to $\CC((K_n)_F)$ as 
a graded $\CC$-vector space, provided the grading of the
latter is shifted by $m$. To verify the claim, and for 
the ease of notation, we let $F = \Sigma_n$ (the general 
case being similar). Consider the linear system of 
parameters (\ref{eq:lsop}) for $\CC[\Delta(K_n)]$. Since
setting all variables $x_v$ associated to vertices $v$ 
not in $K_n$ equal to zero produces the corresponding 
system of parameters for $\CC[K_n]$, the natural 
projection map $\CC[\Delta(K_n)] \to \CC[K_n]$ induces
a well defined linear surjection $\pi: \CC(\Delta(K_n)) 
\to \CC(K_n)$. Since, by \cite[Theorem~2.5]{StaCCA}, the
images of the elements of the basis $\bB_{\Sigma_n}$ form 
a basis of $\CC(K_n)$, the map $\pi$ restricts to an 
isomorphism $W_{\Sigma_n} \to \CC(K_n)$ and the proof of
the claim follows. Next, we observe that $\fS_n$ acts 
(transitively) on the set of faces of $\Sigma_n$ of given 
codimension $m$. By our assumptions, $\fS_n$ also acts on 
the set of restrictions of $K_n$ to these faces and hence 
on the set of restrictions of $\Delta(K_n)$ to the 
corresponding facets of 
$\Delta(\Sigma_n)$, mapping shelling orders to shelling 
orders and hence bases $\bB_F$ to other bases (not 
necessarily the ones we have chosen) of the spaces $W_F$. 
As a result, each subspace $W_{n,m}$ of $\CC(\Delta(K_n))$
is $\fS_n$-invariant and  
\[ W_{n,m} \ \cong_{\fS_n} \ \bigoplus_F \, \CC((K_n)_F) \]
where, essentially by the definition of induction, the 
direct sum on the right is isomorphic to the 
$\fS_n$-representation which is induced from the 
$\fS_m$-representation $\CC(K_m)$. By standard properties 
of Frobenius characteristic, we infer that
\[ \ch ((W_{n,m})_j) (\bx) \ = \ h_{m-n}(\bx) \cdot \ch 
   (\CC(K_m)_j) (\bx) \]
for $0 \le j \le m$. By Equation~(\ref{eq:direct}) and 
our previous discussion, we also have 
\[ \ch (\CC(\Delta(K_n))_j) (\bx) \ = \ \sum_{m=n-j}^n 
   \ch ((W_{n,m})_{j-n+m}) (\bx) \]
and hence 
\[ \ch (\CC(\Delta(K_n))_j) (\bx) \ = \ \sum_{m=n-j}^n 
   h_{m-n}(\bx) \cdot \ch (\CC(K_m)_{j-n+m}) (\bx) \]
for $0 \le j \le n$, which is equivalent to the proposed
formula.
\qed

\begin{example} \rm
Let $K_n$ be the trivial triangulation of $\Sigma_n$.
Then, $\Delta(\Sigma_n)$ is the boundary complex of the
$n$-dimensional cross-polytope and 
Proposition~\ref{prop:athaH} implies that 
\[ 1 \, + \, \sum_{n \ge 1} z^n \, \sum_{j=0}^n \ch 
   (\sigma_{n,j}) (\bx) t^j \ = \ H(\bx; z) H(\bx; tz), \]
where $\sigma_n = \oplus_{j=0}^n \, \sigma_{n,j}$ is the 
graded $\fS_n$-representation on $\CC(\Delta(\Sigma_n))$
induced by the $\fS_n$-action on $\Delta(\Sigma_n)$.
Equivalently, $\sigma_{n,j}$ is isomorphic to the 
permutation $\fS_n$-representation defined by the standard
$\fS_n$-action on the set of $j$-element subsets of $[n]$ 
(a fact which can be shown directly too). The graded 
representation $\sigma_n$ is an $\fS_n$-equivariant 
analogue of $h(\Delta(\Sigma_n),t) = (1+t)^n$ but clearly,
the polynomial $\sum_{j=0}^n \sigma_{n,j} t^j$ cannot be 
$\gamma$-positive for $n \ge 2$. Hence, the equivariant 
Gal phenomenon fails for this action. A similar 
example was given in \cite[Section~5]{SW17+}. \qed
\end{example}

\subsection{An application}

A combinatorial interpretation of the coefficients in 
the $\gamma$-expansion of $\widetilde{B}^+_n (t)$ and 
$\widetilde{B}^{-}_n (t)$ which is similar to the one 
provided for $\widetilde{A}_n (t)$ in \cite{SW17+} 
and to others known for $A_n(t)$ (see, for instance, 
\cite[Theorem~2.1]{Ath17}) can be deduced as follows. 
For a signed permutation 
$w \in \ZZ_2 \wr \fS_n$, we denote by $\Des_B(w)$ the
set of indices $i \in [n]$ for which either $w(i)$ is 
positive and $w(i) > w(i+1)$, or else both $w(i)$ and 
$w(i+1)$ are negative and $|w(i)| > |w(i+1)|$, where 
$w(n+1) := 0$. The set $\Asc_B(w)$ is defined as the 
complement of $\Des_B(w)$ in $[n]$.

\begin{proposition} \label{prop:mainC} 
For every positive integer $n$,
\begin{eqnarray}
\widetilde{B}^+_n (t) & = &  
\sum_{i=0}^{\lfloor n/2 \rfloor} 
\widetilde{\gamma}^+_{n,2,i} \, t^i (1+t)^{n-2i} 
\label{eq:binomAn2+gamma} \\
& & \nonumber \\ 
\widetilde{B}^-_n (t) & = &  
\sum_{i=1}^{\lfloor (n+1)/2 \rfloor} 
\widetilde{\gamma}^-_{n,2,i} \, t^i (1+t)^{n+1-2i}, 
\label{eq:binomAn2-gamma}
\end{eqnarray}
where

\begin{itemize}
\itemsep=0pt
\item[$\bullet$] 
$\widetilde{\gamma}^+_{n,2,i}$ is equal to the 
number of signed permutations $w \in \ZZ_2 \wr 
\fS_n$ for which $\Des_B(w)$ is a set of $i$ 
elements, no two consecutive, and does not 
contain $n$,
 
\item[$\bullet$] 
$\widetilde{\gamma}^-_{n,2,i}$ is equal to the 
number of signed permutations $w \in \ZZ_2 \wr 
\fS_n$ for which $\Des_B(w)$ is a set of $i$ 
elements, no two consecutive, and contains $n$.
\end{itemize}
\end{proposition}

For the first few values of $n$, the numbers 
$\widetilde{\gamma}^{\pm}_{n,2,i}$ appear in the 
expansions 
\[  \widetilde{B}^+_n (t) \ = \ \begin{cases} 
    1 + t, & \text{if $n=1$} \\
    (1 + t)^2 + 3t, & \text{if $n=2$} \\
    (1 + t)^3 + 16t(1 + t), & \text{if $n=3$} \\
    (1 + t)^4 + 61t(1 + t)^2 + 57t^2, & 
                            \text{if $n=4$} \\
    (1 + t)^5 + 206t(1 + t)^3 + 743t^2(1 + t), 
                              & \text{if $n=5$} 
    \end{cases} \]
and
\[  \widetilde{B}^-_n (t) \ = \ \begin{cases} 
    t, & \text{if $n=1$} \\
    3t(1 + t), & \text{if $n=2$} \\
    7t(1 + t)^2 + 11t^2, & \text{if $n=3$} \\
    15t(1 + t)^3 + 98t^2(1 + t), & 
                           \text{if $n=4$} \\
    31t(1 + t)^4 + 577t^2(1 + t)^2 + 361x^3, 
                            & \text{if $n=5$}. 
    \end{cases} \]

\medskip
For the proof, we need some explicit combinatorial 
formula demonstrating the (nonsymmetric) 
$\gamma$-positivity of the coefficient of $z^n$ in 
(\ref{eq:ereesfn}). Conveniently, such a formula has 
been provided in \cite[Corollary~4.7]{Ath18}. We set 
$\Des^\ast(w) = \Des_B(w) \sm \{n\}$ for $w \in \ZZ_2 
\wr \fS_n$ and refer the reader to 
\cite[Section~2]{Ath18} for the definition of the 
signed quasisymmetric function $F_{\sDes(w)} (\bx, 
\by)$, since we will only need here the fact that 
the latter specializes to 
$F_{\Des^\ast(w)}(\bx)$ when setting $\bx = \by$. By 
adding the two equations considered in 
\cite[Corollary~4.7]{Ath18} and setting $\bx = \by$, 
we get 
\begin{eqnarray} \label{eq:rees}
{\displaystyle \frac{(1-t) E(\bx; z) E(\bx; tz)^2} 
{E(\bx; tz)^2 - tE(\bx; z)^2} \ = \ 1}
& + & {\displaystyle \sum_{n \ge 1} z^n
\sum_{i=0}^{\lfloor n/2 \rfloor} 
\widetilde{\gamma}^+_{n,2,i}(\bx) \, t^i 
(1 + t)^{n-2i}} \nonumber \\ & & \nonumber \\
& + & {\displaystyle \sum_{n \ge 1} z^n
\sum_{i=1}^{\lfloor (n+1)/2 \rfloor} 
\widetilde{\gamma}^{-}_{n,2,i}(\bx) \, t^i 
(1 + t)^{n+1-2i}},
\end{eqnarray}
where
\[ \widetilde{\gamma}^+_{n,2,i}(\bx) \ := \ 
   \sum_{w} F_{\Des^\ast(w)}(\bx), \]
\[ \widetilde{\gamma}^{-}_{n,2,i}(\bx) \ := \ 
   \sum_{w} F_{\Des^\ast(w)}(\bx) \]
and $w \in \ZZ_2 \wr \fS_n$ runs through all signed 
permutations for which $\Asc_B(w^{-1})$ is a set of
$i$ elements, no two consecutive, and does not contain
(respectively, contains) $n$.

\medskip
\noindent
\emph{Proof of Proposition~\ref{prop:mainC}.} 
The defining Equation~(\ref{eq:defbinomAnr}) can be 
rewritten as
\[ \sum_{n \ge 0} \widetilde{A}_{n,r} (t) \, 
\frac{z^n}{n!} \ = \ \left( \, \sum_{n \ge 0} A_{n,r} 
(t) \, \frac{z^n}{n!} \right) \left( \, \sum_{n \ge 0} 
t^n \, \frac{z^n}{n!} \right) \ = \ 
\left( \, \sum_{n \ge 0} A_{n,r} (t) \, 
\frac{z^n}{n!} \right) e^{tz}. \]
From these equalities and \cite[Theorem~20]{Stei94} we 
get
\begin{equation} 
\label{eq:binomAnrexpgen}
\sum_{n \ge 0} \widetilde{A}_{n,r} (t) \, 
\frac{z^n}{n!} \ = \ \frac{(1-t) e^{(rt+1)z}}
    {e^{rtz} - te^{rz}}
\end{equation}
and, in particular,
\begin{equation} 
\label{eq:binomBnexpgen}
\sum_{n \ge 0} \widetilde{B}_n (t) \, 
\frac{z^n}{n!} \ = \ \frac{(1-t) e^{(2t+1)z}}
    {e^{2tz} - te^{2z}}.  
\end{equation}
On the other hand, applying ${\rm ex}^\ast$ to 
(\ref{eq:rees}) gives
\begin{eqnarray*} 
{\displaystyle \frac{(1-t) e^{(2t+1)z}}
    {e^{2tz} - te^{2z}} \ = \ 1}
& + & {\displaystyle \sum_{n \ge 1} \left( 
\, \sum_{i=0}^{\lfloor n/2 \rfloor} 
\widetilde{\gamma}^+_{n,2,i} \, t^i 
(1 + t)^{n-2i} \right) \frac{z^n}{n!}}
\\ & & \\
& + & {\displaystyle \sum_{n \ge 1} \left( 
\, \sum_{i=1}^{\lfloor (n+1)/2 \rfloor} 
\widetilde{\gamma}^{-}_{n,2,i} \, t^i 
(1 + t)^{n+1-2i} \right) \frac{z^n}{n!}},
\end{eqnarray*}
where $\widetilde{\gamma}^+_{n,2,i}$ (respectively,
$\widetilde{\gamma}^{-}_{n,2,i}$) stands for the 
number of signed permutations $w \in \ZZ_2 \wr \fS_n$ 
for which $\Asc_B(w)$ is a set of $i$ elements, no 
two consecutive, and does not contain (respectively, 
contains) $n$. Equating the coefficients of $z^n/n!$ 
on both sides, in view of (\ref{eq:binomBnexpgen}), 
we get 
\[ \widetilde{B}_n (t) \ = \ 
   \sum_{i=0}^{\lfloor n/2 \rfloor} 
   \widetilde{\gamma}^+_{n,2,i} \, t^i (1 + t)^{n-2i} 
   \ \, + \sum_{i=1}^{\lfloor (n+1)/2 \rfloor} 
   \widetilde{\gamma}^{-}_{n,2,i} \, t^i 
   (1 + t)^{n+1-2i}. \]
Since there is an obvious involution on $\ZZ_2 \wr 
\fS_n$ which exchanges the set-valued statistics 
$\Asc_B$ and $\Des_B$, the expression just 
obtained for $\widetilde{B}_n (t)$ is equivalent 
to the statement of the proposition.
\qed

\begin{problem}
Find interpretations of  
$\widetilde{\gamma}^{\pm}_{n,r,i}$ for all $r \ge 2$, 
similar to those provided by 
Proposition~\ref{prop:mainC} in the case $r=2$. 
\end{problem}

\section{Equivariant analogues of colored Eulerian 
         polynomials}
\label{sec:eAnr}

Given the results of the previous two sections, 
it is natural to inquire about the existence of 
a well-behaved $\fS_n$-equivariant analogue of 
$A_{n,r}(t)$. 

Consider the $r$th dilate $P = \{ (x_1, x_2,\dots,x_n) 
\in \RR^n: 0 \le x_i \le r\}$ of the standard 
unit $n$-dimensional cube. Although one can use the 
machinery of Section~\ref{sec:fans} to obtain such 
an analogue, via Steingr\'imsson's interpretation 
\cite[Theorem~32]{Stei94} of $A_{n,r}(t)$ as the 
$h$-polynomial of a unimodular triangulation 
of $P$, it is perhaps more straightforward to 
employ Stapledon's equivariant Ehrhart theory 
\cite{Stap11} instead. The $h^\ast$-polynomial 
of $P$ satisfies 
\[ \sum_{k \ge 0} (rk+1)^n \, t^k \ = \
   \frac{h^\ast(P, t)}{(1-t)^{n+1}} \]
and hence, in view of \cite[Theorem~17]{Stei94}, 
we have $h^\ast(P, t) = A_{n,r}(t)$. The symmetric group
$\fS_n$ acts on $\RR^n$ be permuting coordinates, as 
usual, and leaves $P$ invariant. As explained in 
Section~\ref{sec:stap}, denoting by $\varphi_{n,r,j}$ 
the coefficient of $t^j$ in $\varphi^\ast_P(t)$ we 
have an $\fS_n$-equivariant analogue $\varphi_{n,r} = 
\oplus_{j \ge 0} \, \varphi_{n,r,j}$ of the 
$h^\ast$-polynomial $A_{n,r}(t)$ of $P$. Noting that 
$\chi_{kP}(w) = (rk+1)^{c(w)}$, where $c(w)$ is the 
number of cycles of $w$, Equation~(\ref{eq:fSn-phi*}) 
and the definition of Frobenius characteristic imply 
that 
\[ \sum_{j \ge 0} \ch (\varphi_{n,r,j}) (\bx) t^j
   \ = \ (1-t) \, \frac{1}{n!} \sum_{w \in \fS_n} 
	\sum_{k \ge 0} \, (rk+1)^{c(w)} t^k \, 
	\prod_{i \ge 1} \, (1 - t^{\lambda_i(w)}) \, 
	p_{\lambda_i(w)} (\bx). \]
A computation analogous to the one in the proof of 
Proposition~\ref{prop:eGamma-nr} results in the formula
\begin{equation} \label{eq:defphi-nr}
  1 \, + \, \sum_{n \ge 1} z^n \, \sum_{j \ge 0} 
  \ch (\varphi_{n,r,j}) (\bx) t^j \ = \ 
	\frac{(1-t) H(\bx; z) H(\bx; tz)^{r-1}} 
	{H(\bx; tz)^r - tH(\bx; z)^r}.
\end{equation}

For $r=2$, the right-hand side is left invariant after
replacing $z$ with $tz$ and $t$ with $1/t$ and hence 
$\sum_{j \ge 0} \varphi_{n,2,j} t^j = \sum_{j=0}^n 
\varphi_{n,2,j} t^j$ is palindromic, with center of 
symmetry $n/2$. The following stronger statement can
be derived from results of previous sections.

\begin{proposition} \label{prop:phin2} 
The polynomial $\sum_{j=0}^n \varphi_{n,2,j} t^j$ is 
$\gamma$-positive for every positive integer $n$.
\end{proposition}
\begin{proof}
Using Equation~(\ref{eq:tphi-n}), due to Shareshian 
and Wachs~\cite{SW17+}, and Theorem~\ref{thm:psi-nr}
for $r=2$, we find that 

\begin{eqnarray*} 
\frac{(1-t) H(\bx; z) H(\bx; tz)} 
	{H(\bx; tz)^2 - tH(\bx; z)^2}
& = & \frac{(1-t) H(\bx; z) H(\bx; tz)} 
	{H(\bx; tz) - tH(\bx; z)} \cdot 
\frac{H(\bx; tz) - tH(\bx; z)} 
	{H(\bx; tz)^2 - tH(\bx; z)^2} \\ & & \\
& = & \left( 1 \, + \, \sum_{n \ge 1} z^n \, 
      \sum_{i=0}^{\lfloor n/2 \rfloor}    
      \widetilde{\gamma}_{n,i} (\bx) \, 
      t^i (1+t)^{n-2i} \right) \cdot \\ & & \\
&   & \left( 1 \, + \, \sum_{n \ge 1} z^n \, 
      \sum_{i=1}^{\lfloor n/2 \rfloor}    
      \xi^+_{n,2,i} (\bx) \, t^i (1+t)^{n-2i} \right)
\end{eqnarray*}

\medskip
\noindent
for some Schur-positive symmetric functions 
$\widetilde{\gamma}_{n,i} (\bx)$ and $\xi^+_{n,2,i} 
(\bx)$. Computing the coefficient of $z^n$ on the 
right-hand side shows that
\[ \sum_{j=0}^n \ch (\varphi_{n,2,j}) (\bx) t^j \ = \ 
   \sum_{k+\ell=n} \sum_{i, j} \, \widetilde{\gamma}_{k,i} 
   (\bx) \, \xi^+_{\ell,2,j} (\bx) \,  t^{i+j} 
   (1+t)^{n-2i-2j}, \]
where $\widetilde{\gamma}_{n,i}(\bx) = \xi^+_{n,2,j} 
(\bx) := 1$ for $n=i=j=0$. Given that sums and products 
of Schur-positive symmetric functions are Schur-positive, 
this formula implies that the left-hand side is Schur 
$\gamma$-positive for every $n$ and the proof follows. 
\end{proof}

For $r \ge 3$, proceeding as with $\psi_{n,r}$ and 
$\tilde{\varphi}_{n,r}$, we find that there is an 
isomorphism 
\[ \varphi_{n,r} \, \cong_{\fS_n} \, \varphi_{n,r}^+ 
\, \oplus \, \varphi_{n,r}^{-} \]
of graded $\fS_n$-representations, where 
$\varphi_{n,r}^+ = \oplus_{j \ge 0} \, 
\varphi_{n,r,j}^+$ and $\varphi_{n,r}^{-} = 
\oplus_{j \ge 0} \, \varphi_{n,r,j}^{-}$ are such that

\begin{itemize}
\itemsep=0pt

\item
$\sum_{j \ge 0} \, \varphi_{n,r,j}^+ t^j$ is a palindromic
polynomial, with center of symmetry $n/2$, and 

\item
$\sum_{j \ge 0} \, \varphi_{n,r,j}^{-} t^j$ is a
palindromic polynomial, with center of symmetry $(n+1)/2$
and zero constant term
\end{itemize}

\noindent
and
\begin{eqnarray*}
1 \, + \, \sum_{n \ge 1} z^n \, \sum_{j=0}^n \ch 
(\varphi_{n,r,j}^+) (\bx) t^j & = & 
\frac{H(\bx; z) H(\bx; tz) \left( H(\bx; tz)^{r-2} - 
tH(\bx; z)^{r-2} \right)} {H(\bx; tz)^r - tH(\bx; z)^r}
   \\ & & \nonumber \\ 
\sum_{n \ge 1} z^n \, \sum_{j=0}^n \ch 
(\varphi_{n,r,j}^{-}) (\bx) t^j & = & \frac{tH(\bx; z)
H(\bx; tz) \left( H(\bx; z)^{r-2} - H(\bx; tz)^{r-2} 
  \right)} {H(\bx; tz)^r - tH(\bx; z)^r}. 
\end{eqnarray*}

A recent result of Br\"and\'en and Solus 
\cite[Section~3]{BS18} implies that the graded 
dimensions $\sum_{j=0}^n \dim(\varphi_{n,r,j}^+) t^j$ 
and $\sum_{j=0}^n \dim(\varphi_{n,r,j}^{-}) t^j$ of 
$\varphi_{n,r}^+$ and $\varphi_{n,r}^{-}$ are real-rooted, 
hence $\gamma$-positive polynomials for every $n \ge 1$. 
Thus, we expect that the following question  
has an affirmative answer.

\begin{question}
Are $\sum_{j=0}^n \, \varphi_{n,r,j}^+ t^j$ and 
$\sum_{j=0}^n \, \varphi_{n,r,j}^{-} t^j$ 
$\gamma$-positive for all $n \ge 1$ and $r \ge 3$? 
\end{question}

The methods of this paper seem inadequate to address
this question for reasons similar to those explained 
in Remark~\ref{rem:r>2}.

\medskip
\noindent \textbf{Note added in revision}. The 
polynomials $A_{n,r}(t)$, $\widetilde{A}^+_{n,r} (t)$ 
and $\widetilde{A}^{-}_{n,r} (t)$ were shown to be
real-rooted for all values of $n,r$ by Haglund and 
Zhang in~\cite{HZ19}, where additional combinatorial 
interpretations of them appear.

\medskip
\noindent \textbf{Acknowledgments}. The author thanks 
Francisco Santos for providing the short argument in 
the proof of Proposition~\ref{prop:S(G)} (b), the 
anonymous referees for their useful comments on the 
presentation and the content of 
Remark~\ref{rem:referee} and Christina Savvidou for 
helpful discussions.


\begin{thebibliography}{99}
%
\bibitem{Ath12}
C.A.~Athanasiadis,
\textit{Flag subdivisions and $\gamma$-vectors},
Pacific J. Math. {\bf~259} (2012), 257--278.
%
\bibitem{Ath14}
C.A.~Athanasiadis,
\textit{Edgewise subdivisions, local $h$-polynomials 
and excedances in the wreath product 
${\mathbb Z}_r \wr \mathfrak{S}_n$},
SIAM J. Discrete Math. {\bf~28} (2014), 1479--1492.
%
\bibitem{Ath16a}
C.A.~Athanasiadis,
\textit{A survey of subdivisions and local $h$-vectors},
in \textit{The Mathematical Legacy of Richard~P.~Stanley}
(P.~Hersh, T.~Lam, P.~Pylyavskyy, V.~Reiner, eds.),
Amer. Math. Society, Providence, RI, 2016, pp.~39--51.
%
\bibitem{Ath17}
C.A.~Athanasiadis,
\textit{Gamma-positivity in combinatorics and geometry},
S\'em. Lothar. Combin. {\bf~77} (2018), Article B77i, 
64pp (electronic).
%
\bibitem{Ath18}
C.A.~Athanasiadis,
\textit{Some applications of Rees products of posets 
to equivariant gamma-positivity},
{\tt arXiv:1707.08297},
Algebr. Comb. (to appear).
%
\bibitem{BB07}
E.~Bagno and R.~Biagioli,
\textit{Colored-descent representations of 
complex reflection groups $G(r,p,n)$},
Israel J. Math. {\bf~160} (2007), 317--347.
%
\bibitem{BR15}
M.~Beck and S.~Robins,
Computing the Continuous Discretely: Integer-Point
Enumeration in Polyhedra,
Springer, 2015.
%
\bibitem{Bj95}
A.~Bj\"orner,
\textit{Topological methods}, 
in \textit{Handbook of combinatorics}
(R.L.~Graham, M.~Gr\"otschel and L.~Lov\'asz, eds.),
North Holland, Amsterdam, 1995, pp.~1819--1872.
%
\bibitem{BjW05}
A.~Bj\"orner and V.~Welker,
\textit{Segre and Rees products of posets, with 
ring-theoretic applications},
J. Pure Appl. Algebra {\bf~198} (2005), 43--55.
%
\bibitem{BS18}
P.~Br\"and\'en and L.~Solus,
\textit{Symmetric decompositions and real-rootedness},
{\tt arXiv:1808.04141},
Int. Math. Res. Not. (to appear).
%
\bibitem{Br16}
B.~Braun,
\textit{Unimodality Problems in Ehrhart Theory}, 
in \emph{Recent Trends in Combinatorics} 
(A.~Beveridge, J.R.~Griggs, L.~Hogben, G.~Musiker
and P.~Tetali, eds.),
pp.~687--711, Springer, 2016.
%
\bibitem{BR05}
M.~Brun and T.~R\"omer,
\textit{Subdivisions of toric complexes},
J. Algebraic Combin. {\bf~21} (2005), 423--448.
%
\bibitem{CM10}
C.-O.~Chow and T.~Mansour,
\textit{Counting derangements, involutions and 
unimodal elements in the wreath product 
$C_r \wr \mathfrak{S}_n$},
Israel J. Math. {\bf~179} (2010), 425--448.
%
\bibitem{Da78}
V.I.~Danilov,
\textit{The geometry of toric varieties},
Russian Math. Surveys {\bf~33} (1978), 97--154.
%
\bibitem{DRS10}
J.A.~De Loera, J.~Rambau and F.~Santos,
Triangulations: Structures for Algorithms and 
Applications,
Algorithms and Computation in Mathematics {\bf~25},
Springer, 2010.
%
\bibitem{EG00}
H.~Edelsbrunner and D.R.~Grayson,
\textit{Edgewise subdivision of a simplex},
Discrete Comput. Geom. {\bf~24} (2000), 707--719.
%
\bibitem{Ga05}
S.R.~Gal,
\textit{Real root conjecture fails for five- and 
higher-dimensional spheres},
Discrete Comput. Geom. {\bf~34} (2005), 269--284.
%
\bibitem{GS18}
N.~Gustafsson and L.~Solus,
\textit{Derangements, Ehrhart theory and local 
$h$-polynomials},
{\tt arXiv:1807. 05246}.
%
\bibitem{HPPS18}
C.~Haase, A.~Paffenholz, L.C.~Piechnik and F.~Santos,
\textit{Existence of unimodular triangulations --
positive results},
Mem. Amer. Math. Soc. (to appear), 
{\tt arXiv:1405.1687}.
%
\bibitem{HZ19}
J.~Haglund and P.B.~Zhang,
\textit{Real-rootedness of variations of Eulerian 
polynomials},
Adv. in Appl. Math. {\bf~109} (2019), 38--54.
%
\bibitem{HiAC}
T.~Hibi,
Algebraic Combinatorics on Convex Polytopes,
Carslaw Publications, Australia, 1992.
%
\bibitem{LSW12}
S.~Linusson, J.~Shareshian and M.L.~Wachs,
\textit{Rees products and lexicographic shellability},
J. Combinatorics {\bf~3} (2012), 243--276.
%
\bibitem{MMY17}
J.~Ma, S-M.~Ma and Y-N.~Yeh,
\textit{Recurrence relations for binomial-Eulerian polynomials},
{\tt arXiv:1711. 09016}.
%
\bibitem{Mac95}
I.G.~Macdonald,
Symmetric Functions and Hall Polynomials,
second edition,
Oxford University Press, Oxford, 1995.
%
\bibitem{Pet15}
T.K.~Petersen,
Eulerian Numbers,
Birkh\"auser Advanced Texts, Birkh\"auser, 2015.
%
\bibitem{PRW08}
A.~Postnikov, V.~Reiner and L.~Williams,
\textit{Faces of generalized permutohedra},
Doc. Math. {\bf~13} (2008), 207--273.
%
\bibitem{Pro90}
C.~Procesi,
\textit{The toric variety associated to Weyl chambers},
in \textit{Mots} (M.~Lothaire, ed.),
Herm\'es, Paris, 1990, pp.~153--161.
%
\bibitem{SV13}
J.~Schepers and L.~Van Langehoven,
\textit{Unimodality questions for integrally closed 
lattice polytopes},
Ann. Comb. {\bf~17} (2013), 571--589.
%
\bibitem{SW09}
J.~Shareshian and M.L.~Wachs,
\textit{Poset homology of Rees products and 
$q$-Eulerian polynomials},
Electron. J. Combin. {\bf~16} (2) (2009), 
Research Paper 20, 29pp (electronic).
%
\bibitem{SW16}
J.~Shareshian and M.L.~Wachs,
\textit{From poset topology to $q$-Eulerian polynomials
to Stanley's chromatic symmetric functions},
in \textit{The Mathematical Legacy of Richard P.~Stanley}
(P.~Hersh, T.~Lam, P.~Pylyavskyy and V.~Reiner, eds.),
pp.~301--321, Amer. Math. Society, Providence, RI, 2016.
%
\bibitem{SW17+}
J.~Shareshian and M.L.~Wachs,
\textit{Gamma-positivity of variations of 
Eulerian polynomials},
J. Comb. {\bf~11} (2020), 1--33.
%
\bibitem{Sta89}
R.P.~Stanley,
\textit{Log-concave and unimodal sequences in algebra, 
combinatorics, and geometry},
in \textit{Graph Theory and its Applications: East 
and West}, Annals of the New York Academy of Sciences 
{\bf~576}, New York Acad. Sci., New York, 1989, 
pp.~500--535.
%
\bibitem{Sta82}
R.P.~Stanley, 
\textit{Some aspects of groups acting on finite posets}, 
J. Combin. Theory Series A {\bf~32} (1982), 132--161.
%
\bibitem{Sta92}
R.P.~Stanley,
\textit{Subdivisions and local $h$-vectors},
J. Amer. Math. Soc. {\bf~5} (1992), 805--851.
%
\bibitem{StaCCA}
R.P.~Stanley,
Combinatorics and Commutative Algebra,
second edition, Birkh\"auser, Basel, 1996.
%
\bibitem{StaEC1}
R.P.~Stanley,
Enumerative Combinatorics, vol.~1,
Cambridge Studies in Advanced Mathematics {\bf~49},
Cambridge University Press, second edition, 
Cambridge, 2011.
%
\bibitem{StaEC2}
R.P.~Stanley,
Enumerative Combinatorics, vol.~2,
Cambridge Studies in Advanced Mathematics {\bf~62},
Cambridge University Press, Cambridge, 1999.
%
\bibitem{Stap09}
A.~Stapledon,
\textit{Inequalities and Ehrhart $\delta$-vectors},
Trans. Amer. Math. Soc. {\bf~361} (2009), 5615--5626.
%
\bibitem{Stap11}
A.~Stapledon,
\textit{Equivariant Ehrhart theory},
Adv. Math. {\bf~226} (2011), 3622--3654.
%
\bibitem{Stei92}
E.~Steingr\'imsson,
\textit{Permutation statistics of indexed and poset permutations},
Ph.D. thesis, MIT, 1992.
%
\bibitem{Stei94}
E.~Steingr\'imsson,
\textit{Permutation statistics of indexed permutations},
European J. Combin. {\bf~15} (1994), 187--205.
%
\bibitem{Ste92}
J.R.~Stembridge,
\textit{Eulerian numbers, tableaux, and the Betti 
numbers of a toric variety},
Discrete Math. {\bf~99} (1992), 307--320.
%
\bibitem{Ste94}
J.R.~Stembridge,
\textit{Some permutation representations of Weyl groups 
associated with the cohomology of toric varieties},
Adv. Math. {\bf~106} (1994), 244--301.
%
\bibitem{Wa07}
M.~Wachs,
\textit{Poset Topology: Tools and Applications}, 
in \emph{Geometric Combinatorics} 
(E.~Miller, V.~Reiner and B.~Sturmfels, eds.),
IAS/Park City Mathematics Series {\bf~13},
pp.~497--615, Amer. Math. Society, Providence, 
RI, 2007.
%
\end{thebibliography}
\end{document}